\documentclass[a4paper]{amsart}
\usepackage{amsmath}

\newcommand{\R}{\mathbb{R}}

 \newcommand{\C}{\mathbb{C}}
\newcommand{\I}{\mathbb{I}}
\newcommand{\settc}[2]{\bigl\{\,#1 \bigm\vert #2\,\bigr\}}
\newcommand{\cfstr}{{\rm{e}^2}}
\newcommand{\al}{\boldsymbol{\alpha}}
\newcommand{\A}{{\boldsymbol{A}}}

\newcommand{\weakly}{\rightharpoonup}
\newcommand{\fw}{U_{_{\text{FW}}}}

\DeclareMathOperator{\RE}{Re}
\DeclareMathOperator{\spann}{span}

\newtheorem{thm}[equation]{Theorem}

\newtheorem{prop}[equation]{Proposition}
\newtheorem{lem}[equation]{Lemma}

\theoremstyle{remark}
\newtheorem{rem}[equation]{Remark}

\numberwithin{equation}{section}

\title[Solitary wave  of the Maxwell-Dirac equations]{A normalized  solitary wave  solution  of 
\\
the Maxwell-Dirac  Equations}

\author{Margherita Nolasco} 

\email[Nolasco]{nolasco@univaq.it}

\address[Nolasco]{Dipartimento di Ingegneria e Scienze
dell'informazione e Matematica\\
Universit\`a dell'Aquila\\
via Vetoio, Loc.~Coppito\\
67010 L'Aquila (AQ) Italia}

\thanks{Research partially supported by MIUR grant PRIN 2015
2015KB9WPT, ``Variational methods, with applications to problems in
mathematical physics and geometry''.}

\begin{document}

\begin{abstract} 
 We prove the existence of a $L^2$-normalized  solitary wave solution for the Maxwell-Dirac  equations   in (3+1)-Minkowski  space.  In addition, for  the Coulomb-Dirac  model, describing  fermions  with attractive Coulomb interactions in the mean-field limit, we prove the existence of   the (positive) energy minimizer.   \end{abstract} 

\maketitle


\section{Introduction and main results}

	The   Lagrangian  for a   charged,  spin-$\frac{1}{2}$ relativistic  particle  (here $\hbar = c =1$)  interacting with its own  electromagnetic field   is    given by
	
	\begin{equation*}
	\mathcal{L} = \bar{\Psi} ( i \gamma^{\mu} D_{\mu} - m) \Psi - \frac{1}{16 \pi} F_{\mu \nu} F^{\mu \nu}, 
	\end{equation*}
	
where we use the four-vector notations, $\mu, \nu \in \{ 0,1,2,3 \}$ and repeated index summation convention,  with metric tensor $g^{\mu \nu} = \text{diag} \{ 1, -1-1,-1 \}$ used to lower or raise the Lorentz indices. 
$ \gamma^{\mu}$  are the $4 \times    4$   Dirac matrices given by 
  $\gamma^{0} = 
	{\scriptstyle \begin{pmatrix} 
		{\mathbb{I}}_{_{2}} & 0 \\
		0 & - \mathbb{I}_{_2}
	\end{pmatrix} } $ and 
	$ \gamma^{k} =
	{\scriptstyle\begin{pmatrix} 
		0 & \sigma^k \\
		- \sigma^k & 0
	\end{pmatrix} }$,  $k=1,2,3$, 
and $\mathbf{\sigma}_k$ are the $2\times 2$- Pauli  matrices
\begin{equation*}
\sigma_1 = {\scriptstyle \begin{pmatrix} 
	0 & 1 \\
		1 & 0
	\end{pmatrix} } , \quad \sigma_2 = {\scriptstyle\begin{pmatrix} 
	0 & -i \\
		i & 0
	\end{pmatrix} }, \quad \sigma_3 = {\scriptstyle\begin{pmatrix} 
	1 & 0 \\
		-1& 0
	\end{pmatrix} }. 
\end{equation*}

 $\Psi $ is the   Dirac  spinor taking values in $\C^4$ and ${\bar {\Psi }} =  \Psi ^{\dagger }\gamma ^{0}$ is  the \ Dirac adjoint, with $\Psi ^{\dagger }$ the hermitian conjugate of $\Psi$; $m >0$ is the particle's mass, 
 $D_{\mu } = \partial _{\mu } + i  {\rm e} A_{\mu }  $  is the gauge covariant derivative, with  $ {\rm e}$  the  particle's charge ($e <0$ for the electron) and 
$A^{\mu } $ is the  electromagnetic 4-vector potential. $F_{\mu \nu } =\partial _{\mu }A_{\nu }-\partial _{\nu }A_{\mu } $ is the    electromagnetic tensor field.

	The   Euler-Lagrange equations  in the  Lorenz gauge ($ \partial _{\mu } A^{\mu} = 0$)  are given by the   Maxwell-Dirac equations
	
	\begin{equation} 
	\label{eq:MD_full}
	\tag{MD}
	\begin{cases}
	(i \gamma^{\mu}  \partial _{\mu } - e \, \gamma^{\mu} A_{\mu } )\Psi   - m \Psi  =  0 \\
	\partial_{\nu}   \partial^{\nu} A^{\mu} =  4 \pi \, j^{\mu}
	\end{cases}
	\end{equation}
	where $j^{\mu} =  e  \, \bar{\Psi} \gamma^{\mu}  \Psi $ is the  conserved Dirac current  ($ \ \partial_{\mu} j^{\mu} = 0$). 
 We look for   solutions of  \eqref{eq:MD_full}  stationary in time, localized  and $L^2$-normalized in space, called {\it solitary waves}, and which can be seen as     representations  of  the {\it extended particles}. 
 Numerical evidence of the existence of solitary wave solutions of  \eqref{eq:MD_full} was obtained in \cite{Lisi95}. The first  proof of the existence of solitary waves  is given by using variational methods (a linking argument) by  M. Esteban  V. Georgiev and E. S\'{e}r\'{e} in \cite{GeorgievEstebanSere96}.  They proved the existence of stationary solutions $ \Psi(t,x)  = {\rm e}^{ i  \omega t } \psi (x)$, for any $\omega \in (0, m)$, with $\psi$ smooth, and exponentially decreasing at infinity together with all its derivatives. This result was later generalized to   any $\omega \in (-m,m)$ in \cite{Abenda98}, using an axial symmetry ansatz on the class of solutions. Recently in \cite{ComechStuart2018} the authors prove the existence of solitary waves  using a perturbative approach. In fact they prove the existence of  a small amplitude stationary solution  which bifurcates (via Implicit Function theorem) from  the ground state of the   Choquard's equation (see \cite{Lieb77}).
  Let us remark that  both the variational approach used in 
\cite{GeorgievEstebanSere96} (and also in \cite{Abenda98})  and  the perturbative approach used in  \cite{ComechStuart2018}  do not  provide    solutions with  prescribed $L^2$-norm.  Aim of this paper is to find one such   $L^2$-normalized solution. We use a different variational characterization   for critcal points of the energy functional, inspired by the one used to characterize  the first eigenvalue of the Dirac operators with Coulomb-type potentials (see e.g. \cite{DolbeaultEstebanSere2000},  also \cite{CZN2019} for an application in the  nonlinear case) and we use of concentration-compactness-type arguments (see \cite{Lions84}).
Note that  in \cite{CZN2019}  the presence   of  an attractive external Coulomb potential  (the dominant focusing  term) allows one   to  recover compactness. 

 Let us also quote the article \cite{BuffoniEstebanSere2006} where
the authors study normalized solutions for a different problem
which also  has   a strongly indefinite structure. In that paper the authors use
a penalization method in the spirit of \cite{EstebanSere99}.

Our main result is the following. 

\medskip
 
{\bf Theorem}
	{\it There exists $ \omega \in  (0,m) $ and  $\psi \in H^{1/2}(\mathbb{R}^3; \mathbb{C}^4)$,  with $\|\psi \|^2_{L^{2}} = 1$,  such that 
	\begin{equation}
	\label{eq:solution}
	\begin{cases}
	\Psi(x^{0},x) =  {\rm e}^{ i  \omega x^{0}} \psi (x) \\
	A^{\mu} (x^{0},x) = A^{\mu} (x)  =   e  (\gamma^0  \psi,  \gamma^{\mu}  \psi)_{\C^4}  \ast \frac{1}{|x|}
	\end{cases}
	\end{equation}
	is  a solution of  \eqref{eq:MD_full}.}

	\medskip
	
As already mentioned we prove this result by using a variational characterization    of critical level of the energy functional  introduced  for  the first eigenvalue of Dirac operators. Indeed   
let $A= (A^0, \boldsymbol{A})$, with $\boldsymbol{A}= (A^1, A^2,A^3)$,  be the four -vector potential $A^{\mu}$, clearly $A^0= A_0 $ and $A^k=- A_{k}$, ($k=1,2,3$), 
and let  denote $ \beta  =  \gamma^0 $ and $\boldsymbol{\alpha} = (\alpha^{1} , \alpha^{2} , \alpha^{3} )$, with $ \alpha^{k} = \gamma^0 \gamma^k $ ($k=1,2,3$). 
Then  $(\Psi, A)$ is a  $L^2$-normalized  stationary  solution of \eqref{eq:MD_full}    of the form  \eqref{eq:solution}  if $ (\psi, \omega)$ is a solution of the following (nonlinear)
eigenvalue problem 

\begin{equation}
		\label{eq:eigenvalue_MD}
		\tag{$E_{\omega}$}
		\begin{cases} 
			( i \al \cdot \nabla - m  \beta) \psi  - e A_0  \psi + e \al \cdot  \A  \,  \psi  =  \omega \,  \psi \\
				 A_0(x) =   e \, |\psi|^{2} \ast \frac{1}{|x|}  ; \qquad  \A(x)  =    e   \, ( \psi ,  \al \psi )_{\C^4} \ast \frac{1}{|x|} \\
			\|\psi \|^2_{L^{2}} = 1 .
			\end{cases} 
		\end{equation}

We look for solutions of \eqref{eq:eigenvalue_MD} as (constrained)  critical points of the functional
\begin{equation*}
\mathcal{I}_{_{MD}} (\psi) = \int_{\R^3} (\psi, D\psi )_{\C^4} \, - \frac{e^2}{2} \int_{\R^3 \times \R^3} \frac{\rho_{\psi}(x)\rho_{\psi}(y) - J_{\psi}(x) \cdot J_{\psi}(y)}{|x-y| } dx dy
\end{equation*}
where  $D = i \al \cdot \nabla - m  \beta $   and $\rho_{\psi} = |\psi|^2$ and $J_{\psi} = ( \psi ,  \al \psi )$, 
on   the manifold
\begin{equation*}
\Sigma = \{ \psi \in H^{1/2}(\R^3, \C^4) \, : \, \|\psi \|^2_{L^{2}} = 1 \}.
\end{equation*}
Note also that in the units we choose ($\hbar = c =1$) the coupling constant ${\rm e}^2$ is, as a matter of fact,   the dimensionless
fine structure constant $\frac{e^{2}}{\hbar c} \approx \frac{1}{137}$. 
The functional  $\mathcal{I}_{_{MD}}$ is strongly indefinite and presents a lack of compactness. Indeed, the  operator
$D =  i  \al \cdot \nabla -  m  \beta $  is a first order,   self-adjoint
	operator  on $H^1(\mathbb{R}^3 , \mathbb{C}^4)$   with purely absolutely continuous spectrum given by
  \begin{equation*}
  \sigma(D) = (-\infty, -m] \cup [m, +\infty). 
  \end{equation*}
  		Let  $\Lambda_{\pm}(D)  $ be the  two infinite rank  orthogonal projectors on  the positive/negative 
    energies subspaces, then  
    \begin{equation*}
	D \Lambda_{\pm}(D)  = \Lambda_{\pm}(D)  D =  \pm  \sqrt{- \Delta + m}  \, \Lambda_{\pm}(D) = \pm  \Lambda_{\pm}(D) \,  \sqrt{- \Delta + m} , 
	    \end{equation*}
	   
hence for   $\psi \in H^{1/2}(\mathbb{R}^3 ; \mathbb{C}^4 )$  the operator form is given by
  
  \begin{equation*}  
  \int_{\R^3} (\psi, D \psi )_{\C^4}= \| (- \Delta + m)^{1/4}  \Lambda_{+}(D)  \psi \|^2_{L^2} - \|  (- \Delta + m)^{1/4}  \Lambda_{-}(D)  \psi \|^2_{L^2}
  \end{equation*}
   and we denote $X_{\pm}(D)  =  \Lambda_{\pm}(D)   H^{1/2}(\R^3; \C^4)$.

In fact we prove the existence of  the $L^2$-normalized solitary wave solution of  \eqref{eq:MD_full} by means of the following variational characterization. 

\begin{thm}
\label{thm:minmax}
Let define
\begin{equation*}
E =  \inf_{\substack{ W \subset X_{+}(D) \\ \dim W = 1}} \,
	 \sup_{\substack{ \phi \in W \oplus X_{-}(D)  \\  \|\phi\|_{L^2} = 1}} \mathcal{I}_{_{MD}} (\phi)
\end{equation*}
then $E \in (0,m) $ and it is attained, namely there exists $\psi \in \Sigma$ such that $ \mathcal{I}_{_{MD}} (\psi) = E $. Moreover, 
there exists $\omega \in (0,m) $ (Lagrange multiplier)  such that 
\begin{equation*}
d \mathcal{I}_{_{MD}}(\psi)[h] = 2 \omega \RE \langle \psi | h \rangle_{L^2},  \qquad \forall h \in H^{1/2}(\R^3, \C^4)
\end{equation*}
that is  $(\psi,w)  \in  H^{1/2}(\R^3, \C^4) \times (0,m)$ is a solution of \eqref{eq:eigenvalue_MD} and  
\begin{equation*}
	\begin{cases}
	\Psi(x^{0},x) =  {\rm e}^{ i  \omega x^{0}} \psi (x) \\
	A^{\mu} (x^{0},x) = A^{\mu} (x)  =   e  (\gamma^0  \psi,  \gamma^{\mu}  \psi)_{\C^4}  \ast \frac{1}{|x|}
	\end{cases}
	\end{equation*}
	is  a $L^2$- normalized  solitary wave solution of  \eqref{eq:MD_full}.
	
	In addition, $E$  is  the lowest positive  critical value  of the  functional  $\mathcal{I}_{_{MD}}$ on $\Sigma$. 
\end{thm}

As a byproduct of the proof of Theorem \ref{thm:minmax}, we obtain also  an interesting  result  for the Coulomb-Dirac model,   describing  fermions  with attractive Coulomb interactions and 
that can be viewed as a semiclassical approximation of the (relativistically invariant) {\it polaron} model. We refer to  \cite{ComechZuvkov2013} for a detailed discussion of this model and its solitary waves 
and  to  \cite{Abenda98} for a multiplicity results of (not normalized) stationary solutions.  

Denoting    $H = - i \al \cdot \nabla + m  \beta = - D$, note that this is the operator usually called   the (free) Dirac operator,  clearly $\Lambda_{\pm}(H) = \Lambda_{\mp}(D)$ and  $X_{\pm}(H)  =  X_{\mp}(D)$, then we have the following result.

\begin{thm}
	\label{thm:polaron} 
	 There exists $ \omega \in  (0,m) $ and  $\psi \in H^{1/2}(\mathbb{R}^3; \mathbb{C}^4)$  solution of 
	 
	 \begin{equation}
		\label{eq:eigenvalue_CD}
		\begin{cases} 
			(- i \al \cdot \nabla + m  \beta) \psi  +  e A_0  \psi  =  \omega \,  \psi \\
				 A_0(x) =  -  e \,  |\psi|^{2} \ast \frac{1}{|x|}  \\
			\|\psi \|^2_{L^{2}} = 1 .
			\end{cases} 
		\end{equation}
	 	Moreover, 
	\begin{equation*}
	 \mathcal{I}_{_{CD}} (\psi) =  \inf_{\substack{ W \subset X_{+}(H) \\ \dim W = 1}} \,
	 \sup_{\substack{ \phi \in W \oplus X_{-}(H)  \\  \|\phi\|_{L^2} = 1}} \mathcal{I}_{_{CD}} (\phi) =  E  \in (0,m), 
	\end{equation*}
	 where
	\begin{equation*}
	\mathcal{I}_{_{CD}}(\phi)=  \int_{\R^3} (\phi, H \phi )_{\C^4}   - e^2 \int_{\R^3 \times \R^3} \frac{ |\phi|^2 (x)  |\phi|^2 (y) }{|x-y|} dx dy. 
	\end{equation*}
	In addition, $E$  is  the lowest positive  critical value  of the energy functional  $\mathcal{I}_{_{CD}}$. 
	\end{thm}	
Let us mention a related result obtained in  \cite{FrohlichJonssonLenzmann2007} where the authors prove the existence and orbital stability for  the $L^2$-normalized,  solitary wave solution, minimizer of the energy functional for   the pseudo-relativistic model describing  bosons with attractive Coulomb interactions.

\section{Notation and preliminary results}

From now on we take   $m=1$.  We denote by $ \hat{u}$ or
$\mathcal{F}(u)$ the Fourier transform of $u$, defined by extending the
formula
\begin{equation*} 
	\hat{u}(p) = \frac{1}{(2\pi)^{3/2}} \int_{\mathbb{R}^{3}} e^{-ip
	\cdot x} u(x) \, dx, \qquad \text{for } \, u \in
	\mathcal{S}(\mathbb{R}^{3}).
\end{equation*}
We denote
\begin{equation*} 
  \langle f | g \rangle_{H^{1/2}} = \int_{\R^3} \sqrt{|p|^2 + 1}
	(\hat{f}(p), \hat{g}(p))	 \, dp
	\end{equation*} 
 the scalar product in $H^{1/2}(\R^3, \C^4)$  with   $(\, \cdot \, ,  \cdot \, )$   the hermitian scalar product in $\C^4$.

Let $H = - i  \al \cdot \nabla +  \beta $   be  the (free) Dirac operator,        
in the (momentum) Fourier space we have 
the multiplication operator $\hat{H} (p) = \mathcal{F}  H
\mathcal{F}^{-1} = \al \cdot p + \beta$ which, for each $p \in
\R^3$, is an Hermitian $4 \times 4$-matrix with eigenvalues
\begin{equation*}
	\lambda_1(p) = \lambda_2(p) = - \lambda_3(p)= - \lambda_4(p) =
	\sqrt{|p|^2 + 1} \equiv \lambda(p).
\end{equation*}

The unitary transformation $U(p)$ which diagonalize $\hat{H}(p)$
is given explicitly by
\begin{align*}
	& U(p) = u_{+} (p) \I_{4} + u_{-} (p) \beta
	\frac{\al \cdot p}{|p|} \\
    &U^{-1}(p) = u_{+} (p) \I_{4} - u_{-} (p) \beta
    \frac{\al \cdot p}{|p|} = U^{\dag}(p)
\end{align*}
with $u_{\pm}(p) = \sqrt{\frac{1}{2}(1 \pm \frac{1}{\lambda(p)})}$. We
have
\begin{equation*}
	U(p)\hat{H}(p)U^{-1}(p) = \lambda(p) \beta = \sqrt{|p|^2 + 1}
	\, \beta.
\end{equation*}

Hence the  two orthogonal projectors $ \Lambda_{\pm}(H) $ on
$L^{2}(\mathbb{R}^{3}, \mathbb{C}^{4})$ are 
given by
\begin{equation}
	\label{eq:projectors}
	\Lambda_{\pm}(H) = \frac{1}{2} \mathcal{F}^{-1}U(p)^{-1} ( \I_{4} \pm
	\beta ) U(p)\mathcal{F}. 
\end{equation}
  We denote $X_{\pm}(H)  =  \Lambda_{\pm}(H)   H^{1/2}(\R^3; \C^4)$. Clearly we have $\Lambda_{\pm}(H) = \Lambda_{\mp}(D)$ and  $X_{\pm}(H)  =  X_{\mp}(D)$.

It may be useful  consider the Foldy-Wouthuysen (FW) transformation (see e.g. \cite{Thaller1992}), namely the
unitary transformation $\fw = \mathcal{F}^{-1} U(p) \mathcal{F}$. Note that 
under the FW transformation the projectors $\Lambda_{\pm}(H) $ become
simply
\begin{equation}
\label{eq:Foldy}
	\Lambda(H)_{\pm}^{^{(FW)}}  = \fw \Lambda_{\pm} (H)  \fw^{-1} = \frac{1}{2} ( \I_{4} \pm \beta). 
\end{equation}
Note that  $\Lambda(D )_{\pm}^{^{(FW)}} =  \frac{1}{2} ( \I_{4} \mp \beta)$.

We consider the smooth functional $\mathcal{I}\colon H^{1/2}(\R^3,\mathbb{C}^4) \to \R $ given by 
\begin{align*}
	\mathcal{I}  (\psi) =  \| \psi_{+} \|^2_{H^{1/2}} - \| \psi_{-}
	\|^2_{H^{1/2}} 
	-   \frac{\cfstr}{2}    \int_{\R^3 \times \R^3} \frac{ \rho_{\psi}
	(x) \rho_{\psi} (y) - J_{\psi}(x)
	\cdot J_{\psi} (y)} {|x -y |} dx dy
\end{align*}
where $ \psi_{\pm} = \Lambda_{\pm}(D) \psi $, $\rho_{\psi} =
|\psi|^2$ and $J_{\psi} = ( \psi , \al \psi )$.

 The Frech\'et derivative
 $ d \mathcal{I} (\phi) : H^{1/2}(\R^3,\mathbb{C}^4) \to \R$ is given by
\begin{align*}
	d \mathcal{I} (\psi) [h] = & 2 \RE \langle \psi_{+} | h_{+}
	\rangle_{H^{1/2}} - 2 \RE \langle \psi_{-} | h_{-} \rangle_{H^{1/2}}  \\
	& - 2 \cfstr \int_{\R^3 \times \R^3 } \frac{ \rho_{\psi} (x) \RE
	(\psi, h )(y) - J_{\psi} (x) \cdot \RE (\psi, \al
	h )(y)} { |x-y| } dx dy
\end{align*}
for  any  $h= h_{+} + h_{-} \in   H^{1/2}(\R^3,\mathbb{C}^4)$,  with $h_{\pm}  \in  X_{\pm}(D) $.

Clearly  $(\psi, \omega) \in H^{1/2}(\R^3,\mathbb{C}^4) \times \R $ is a weak solution of \eqref{eq:eigenvalue_MD}  if and only if
\begin{equation*}
	d\mathcal{I} (\psi) [h]= \omega  \, 2 \RE \,
	\left\langle \psi | h \right\rangle_{{L^2}}, 
	\qquad \forall h \in H^{1/2}(\R^3, \mathbb{C}^4).
\end{equation*} 
Hence we look for (constrained) critical points of $\mathcal{I}$ on   the manifold
 \begin{equation*}
\Sigma = \{ \psi \in H^{1/2}(\R^3, \C^4) \, : \, \|\psi \|^2_{L^{2}} = 1 \}.
\end{equation*}

\begin{rem}
	\label{rem:Hardy}
	Let us recall the following Hardy-type inequalities :
	\begin{description}
		\item[{\it Hardy}]   $\| |x|^{-1} \psi \|^2_{{L^2}} \leq 4 \| \nabla \psi \|^2_{{L^2}} $ for all $\psi \in H^1(\R^3)$; 
			
			\medskip
					
		\item[{\it Kato}]   $	\| |x|^{-\frac{1}{2}} \psi \|^2_{L^2} \leq \gamma_{K} \|(-\Delta)^{1/4} \psi \|^2_{L^2} $ for all $\psi \in H^{1/2}(\R^3) $, with  $\gamma_{K}  = \frac{\pi}{2}$.
	\end{description}
	Let us remark that $ e^2 \gamma_{K}  < \frac{1}{87}$. 
	
	In view of Kato's inequality  for any $\rho \in L^{1}(\mathbb{R}^{3})$ and $\psi \in
	H^{1/2}(\R^3,\mathbb{C}^4) $ we have
	\begin{equation}
	\label{eq:key_estimate}
		 \int_{\R^3 } (\rho \ast \frac{1}{|x|} ) | \psi |^2 (y) dy
		\leq  \gamma_{K}  \|\rho \|_{L^{1}} \|(-\Delta)^{1/4} \psi \|^2_{L^2}.
	\end{equation}
	
	\end{rem}
\begin{rem}
	\label{rem:conti_fourier}
	Since $ \mathcal{F}[ \frac{1}{|x|} ] =  \sqrt{ \frac{2}{\pi} } \frac{1}{ |p|^{2}} $, 
	 for any $f \in L^1 \cap L^{3/2}$ we have that 
	\begin{equation*}
	\label{eq:positive_Fourier}
		\int_{\R^3 \times \R^3} \frac{f(x) \bar{f}(y) }{|x-y|} dx dy  =
		4 \pi  \int_{\R^3} \frac{ |\hat{f}|^2(p) }{|p|^2}\, dp
		\geq 0.
	\end{equation*} 
	
	Hence in particular
	\begin{equation}
		\label{eq:correntePositiva}
		\int_{\R^3 \times \R^3 } \frac{ J_{\psi} (x) \cdot J_{\psi}
		(y) } { |x-y | } dx dy \geq 0 .
	\end{equation} 
Moreover 
	since $ |J_{\psi} | \leq \rho_{\psi}  $ for any $  \psi \in H^{1/2}(\R^3, \C^4)$, we have that
	\begin{equation}
	\label{eq:stimaGES}
		\int_{\R^3 \times \R^3 } \frac{ \rho_{\psi} (x) \rho_{\psi}
		(y) - J_{\psi} (x) \cdot J_{\psi} (y) } { |x-y| } dx
		dy \geq 0.
	\end{equation}
	
\end{rem}

\medskip

Moreover we have the following useful result (see the Appendix for the proof). 
\begin{lem} 
\label{lem:key2} 
For   any $\psi   = \psi_{+} + \psi_{-}  \in  \Sigma $,  let define $w = \frac{\psi_{+}}{ \| \psi_{+} \|_{L^2}}$    we have 
\begin{align*}
\int_{\R^3 \times \R^3}  &  \frac{ \rho_{\psi}
	(x) \rho_{\psi} (y)  -  J_{\psi}
	(x) \cdot  J_{\psi}(y)} {|x -y |}   dx dy  \geq 
	  \int_{\R^3 \times \R^3}   \frac{ \rho_{w}
	(x) \rho_{w} (y) -  J_{w} (x)  \cdot J_{w}(y)} {|x-y|} dx dy \\
	& -  8 \gamma_{K}   ( \|w\|^2_{H^{1/2} }- \|w\|^2_{L^2} )  - 10  \gamma_{K} (   \|\psi_{-} \|^2_{L^2} \|w\|^2_{H^{1/2} }  + \|  \psi_{-} \|^2_{H^{1/2}} ) .
\end{align*}
Moreover, if $v \in H^{1}(\R^3, \C^2)$,  with $\| v \|^2_{L^2} = 1$,  and  $\frac{\psi_{+}}{ \| \psi_{+} \|_{L^2}} = \fw^{-1} \left(  \begin{matrix} 0 \\ v  \end{matrix} \right)  $ we have 
\begin{align*}
\int_{\R^3 \times \R^3}  &  \frac{ \rho_{\psi}
	(x) \rho_{\psi} (y)  -  J_{\psi}(x) \cdot  J_{\psi}(y)} {|x-y|}   dx dy 
 \geq    \int_{\R^3 \times \R^3}   \frac{ \rho_{v}(x)  \rho_{v}(y) } {|x - y|} dx  dy  \\
 & -  8 \gamma_{K}   \| \nabla v \|^2_{L^2}  - 10  \gamma_{K} (   \|\psi_{-}\|^2_{L^2} \|v \|^2_{H^{1/2} }  + \|  \psi_{-} \|^2_{H^{1/2}} ) .
 \end{align*}
\end{lem}

\medskip

Finally we  recall the following convergence result.  Let $ v \in
H^{1/2}$, $f_n, g_n, h_n$ bounded sequences in $H^{1/2}$ such that  one of
them converges weakly to zero in $H^{1/2}$, then we have (see for
example \cite[Lemma 4.1]{CZN2013})
\begin{equation}
	\label{eq:convergenceto0}
	\int_{\R^3 \times \R^3} \frac{ |f_n|(x) |g_n|(x) |v|(y ) |h_n|(y)
	}{|x-y|} dx dy \to 0,   \qquad \text{as} \, n \to + \infty. 
\end{equation}

\section{Maximization problem} 

We  introduce the family of functionals   $\mathcal{I}^{(m)} \colon H^{1/2}(\R^3,
\mathbb{C}^4) \to \R $, with  $m \in (0,1]$,   
\begin{align*}
	\mathcal{I}^{(m)}  (\psi) =  \| \psi_{+} \|^2_{H^{1/2}} - \| \psi_{-}
	\|^2_{H^{1/2}} 
	-  m \frac{\cfstr}{2}    \int_{\R^3 \times \R^3} \frac{ \rho_{\psi}
	(x) \rho_{\psi} (y) - J_{\psi}(x)
	\cdot J_{\psi} (y)} {|x -y |} dx dy,
\end{align*}
where $ \psi_{\pm} = \Lambda_{\pm} (D) \psi $, $\rho_{\psi} =
|\psi|^2$ and $J_{\psi} = ( \psi , \al \psi )$.  Clearly  $\mathcal{I} = \mathcal{I}^{(1)}$.

Our first step will be to maximize the family of  functionals   $\mathcal{I}^{(m)}$  on the space 
\begin{equation*}
	\mathcal{X}_{W} = \settc{ \psi \in  \Sigma }{\psi_{+} \in W },
\end{equation*}    
where  $W  \subset  X_{+}(D)  $ is a 1-dimensional vector space.

The tangent
space of $ \mathcal{X}_{W} $ at some point $\psi \in \mathcal{X}_{W}$
is the set
\begin{equation*}
	T_{\psi} \mathcal{X}_{W} = \settc{h \in W \oplus X_{-} (D)
	}{ \RE \langle \psi |
	h\rangle_{_{L^{2}} }= 0}
\end{equation*} 
and $\nabla_{\mathcal{X}_{W} }\mathcal{I}^{(m)} (\psi)$, the projection of
the gradient $\nabla \mathcal{I}^{(m)} (\psi)$ on
$T_{\psi}\mathcal{X}_{W} $,  is given by
\begin{equation*}
	\RE \langle \nabla_{\mathcal{X}_{W} }  \mathcal{I}^{(m)} (\psi) | h \rangle_{_{H^{1/2}}}= d\mathcal{I}^{(m)} (\psi)[h]  -  2 \omega(\psi) \RE \langle \psi | h\rangle_{_{L^{2}} }
\end{equation*}
for all $ h \in  W \oplus X_{-}(D) $ and $\omega(\psi) \in \mathbb{R}$ is such that
$\nabla_{\mathcal{X}_{W} }\mathcal{I}^{(m)} (\psi) \in T_{\psi} \mathcal{X}_{W}$.

Let us introduce 
\begin{equation*}
	\Sigma_{+} = \settc{ w \in  X_{+}(D) }{ \|w \| ^2_{L^2} = 1}, 
\end{equation*}  
then, from now on,  we characterize the 1-dimesional vector space $W \subset  X_{+}(D)$ as $W = \spann\{ w \}$, with $w \in \Sigma_{+}$.

We begin giving a result on Palais-Smale sequences of $\mathcal{I}^{(m)} $  on $\mathcal{X}_{W}$, in particular we prove that the Palais-Smale condition holds on $\mathcal{X}_{W}$ for  $\mathcal{I}^{(m)} $ at  the positive levels.

\begin{prop}
	\label{prop:PalaisSmale}
For  any $w  \in  \Sigma_{+}$  and  for any $m \in (0,1]$,  let $\{ \psi_{n}  \} \subset  \mathcal{X}_{W}$ be a Palais-Smale sequence
	of  $\mathcal{I}^{(m)} $  on $\mathcal{X}_{W}$, that is
	\begin{equation*}
	 \mathcal{I}^{(m)} (\psi_{n})  \to c \,\, \text{and} \,\,   \| \nabla_{\mathcal{X}_{W}}\mathcal{I}^{(m)} (\psi_{n})  \|_{H^{1/2}} \to 0 , \qquad  \text{as} \quad   n \to + \infty.
	 \end{equation*}
	 Then,  
	
	\begin{itemize}
	
	\item[(i)] $\{ \psi_{n} \}  \subset  \mathcal{X}_{W} $ is bounded in $H^{1/2}$;
	
	\item[(ii)]  $ \omega(\psi_n)$ is bounded and  $  \liminf_{n \to + \infty} \omega(\psi_n)( \| (\psi_{n})_{+} \|^2_{L^2}-   \|( \psi_{n})_{-} \|^2_{L^2} ) > 0 $;

	\item[(iii)] If $c > 0 $ than  $\liminf_{n \to + \infty}   \omega(\psi_n) > 0 $ and $\{ \psi_{n} \} $ is pre-compact in $H^{1/2}$.

	\end{itemize}
\end{prop}

\begin{proof}

(i)  Since $ \{ (\psi_{n})_{+}  \} \subset
	W$ and  $	\| (\psi_{n})_{+} \|^2_{L^2} \leq 1$,  we have $\|(\psi_{n})_{+}\|_{H^{1/2}} \leq  \|w\|_{H^{1/2}}$.
	 In view of \eqref{eq:stimaGES}
		we have
	\begin{equation*} 
		 \mathcal{I}^{(m)} (\psi_{n}) \leq \| (\psi_{n})_{+}	\|^2_{H^{1/2}} - \| (\psi_{n})_{-}  \|^2_{H^{1/2}} 
	\end{equation*}
	hence we get $\| (\psi_{n})_{-} \|^2_{H^{1/2}}  \leq  \|w \|^2_{H^{1/2}}  -   \mathcal{I}^{(m)} (\psi_{n})  \leq \|w \|^2_{H^{1/2}}  - c +  o(1) $.

\medskip

(ii) Since
	\begin{align*}
	 \omega(\psi_{n})	+ o(1) = &  \frac{1}{2} d\mathcal{I}^{(m)} (\psi_{n})[ \psi_{n} ] \\
	 =&  \mathcal{I}^{(m)} (\psi_{n})  -  m  \frac{\cfstr}{2} \int_{\R^3 \times \R^3}
		\frac{\rho_{\psi_n}(x) \rho_{\psi_{n}}(y) -  J_{\psi_n}(x)
		\cdot J_{\psi_{n}}(y) }{ |x-y|}  dx dy 
		 	\end{align*}
		by  \eqref{eq:key_estimate} and \eqref{eq:stimaGES} we have 
			\begin{equation*}
  \mathcal{I}^{(m)} (\psi_{n})  - m  \frac{\cfstr}{2} \gamma_{K}  \|\psi_{n} \|^2_{H^{1/2}}	\leq \omega(\psi_{n}) \leq  \mathcal{I}^{(m)} (\psi_{n}) .	
			\end{equation*}
 Then  since $\{\psi_{n}  \}$ is a  bounded sequence  in $H^{1/2}$  we conclude that   $ \omega(\psi_n)$ is  a  bounded sequence. Moreover,   we have 
	
	\begin{equation*}
		o(1)= d\mathcal{I}^{(m)} (\psi_{n})[ (\psi_{n})_{+} - (\psi_{n})_{-}  ] - 2 \omega(\psi_{n})
	   (	\| (\psi_{n})_{+} \|^2_{L^2} -  \|  (\psi_{n})_{-} \|^2_{L^2})
	\end{equation*}
	and since 
			$\RE (\psi_{+}+ \psi_{-}, \psi_{+}- \psi_{-})=|\psi_{+}|^{2} - |\psi_{-}|^{2} $
again by  \eqref{eq:key_estimate}  and \eqref{eq:stimaGES} we get 
		\begin{align*}
		& \omega(\psi_{n}) (\|(\psi_{n})_{+}  \|^2_{L^2} -  \|(\psi_{n})_{-} \|^2_{L^2} )+ o(1) =
		\| (\psi_{n})_{+} \|^2_{H^{1/2}} + \| (\psi_{n})_{-} \|^2_{H^{1/2}}  \\
		& \qquad - m  \cfstr \int_{\R^3 \times \R^3} \frac{\rho_{\psi_n}(x) ( |(\psi_{n})_{+} |^2 - |(\psi_{n})_{-} |^2)(y)}{ |x-y|}  dx dy \\
		& \qquad 
		+ m  \cfstr \int_{\R^3 \times \R^3} \frac{J_{\psi_n}(x)
		\cdot( \RE ((\psi_{n})_{+} , \al  (\psi_{n})_{+} )  - \RE ((\psi_{n})_{-} , \al  (\psi_{n})_{-} ) (y)}{ |x-y|} dx dy  \\
		&\geq  \| (\psi_{n})_{+} \|^2_{H^{1/2}} + \| (\psi_{n})_{-} \|^2_{H^{1/2}} 
	 -  2 m  \cfstr \int_{\R^3 \times \R^3}
		\frac{\rho_{\psi_n}(x) |(\psi_{n})_{+} |^2 (y)}{ |x-y|} dx dy \\
		&\geq (1 - 2\cfstr
		\gamma_{K})  \| (\psi_{n})_{+} \|^2_{H^{1/2}} +  \|(\psi_{n})_{-}  \|^2_{H^{1/2}} >  1 - 2\cfstr
		\gamma_{K}.
		\end{align*}
		
		\medskip
		
(iii)  If $c >0$ clearly $\| (\psi_{n})_{+}  \|^2_{H^{1/2}} \geq \| (\psi_{n})_{-}  \|^2_{H^{1/2}} $ for $n $ sufficiently large, and  by  \eqref{eq:key_estimate} we have 
\begin{align*}
 \omega(\psi_{n}) \|(\psi_{n})_{+}  \|^2_{L^2} & + o(1) =   \frac{1}{2} d\mathcal{I}^{(m)} (\psi_{n})[ (\psi_{n})_{+}  ]\\
		= & \|(\psi_{n})_{+} \|^2_{H^{1/2}}    - m  \cfstr \int_{\R^3 \times \R^3}
		\frac{\rho_{\psi_n}(x)  \RE (\psi_{n}, (\psi_{n})_{+} )(y)}{ |x-y|}  dx dy \\
		&  +m  \cfstr \int_{\R^3 \times \R^3} \frac{J_{\psi_n}(x)\cdot   \RE (\psi_{n}, \al (\psi_{n})_{+} )(y)}{ |x-y|} dx  dy \\
		\geq & \| (\psi_{n})_{+}  \|^2_{H^{1/2}} 	 -  2 \cfstr  \gamma_{K}  \|\psi_{n} \|_{H^{1/2}}  \|(\psi_{n})_{+} \|_{H^{1/2}}\\
			\geq & \| (\psi_{n})_{+}  \|^2_{H^{1/2}} 	 -  4 \cfstr  \gamma_{K}   \| (\psi_{n})_{+}  \|^2_{H^{1/2}} 
		\geq   (1 - 4\cfstr \gamma_{K})  \| ( \psi_{n})_{+}  \|^2_{H^{1/2}} .
		\end{align*}
		Hence we get  $ \omega(\psi_{n})  \geq (1 - 4\cfstr \gamma_{K}) + o(1) $.
 \medskip

 Since $ \{ (\psi_{n})_{+} \}  \subset W$ clearly $ (\psi_{n})_{+}   \to   \psi_{+} $ in $H^{1/2} $ (up to subsequence), moreover $  (\psi_{n})_{-}  \weakly  \psi_{-} $ weakly  in $H^{1/2} $ (up to subsequence) and since   $  \liminf_{n \to + \infty} \omega(\psi_n) > 0 $,  we have  

\begin{align*}
		& o(1)=  - \frac{1}{2} d\mathcal{I}^{(m)} (\psi_{n})[ (\psi_{n})_{-}  -  \psi_{-} ] + \omega(\psi_{n}) 
		  \|(\psi_{n})_{-} -  \psi_{-} \|^2_{L^2} \\
		 \geq  &   \| (\psi_{n})_{-} -   \psi_{-} \|^2_{H^{1/2} }  +  m  \cfstr \int_{\R^3 \times \R^3}
		\frac{(\rho_{\psi_n} - |J_{\psi_n}|)(x)  |(\psi_{n})_{-}   -   \psi_{-} |^2 (y)}{ |x-y|}  dx dy  + o(1)\\
		 \geq  &   \| (\psi_{n})_{-}  -   \psi_{-} \|^2_{H^{1/2} } + o(1). 
			\end{align*}
	and we may conclude that also $ ( \psi_{n})_{-}  \to   \psi_{-} $ strongly  in $H^{1/2} $.

	\end{proof}
	It turns out that all the critical points of $ \mathcal{I}^{(m)} $ on $\mathcal{X}_W $ at positive levels are strict local maxima.  More precisely we have the following result.
        
\begin{prop}
	\label{prop:concavity}
For any $m \in (0,1]$ let $\psi^{(m)} \in H^{1/2}(\R^3; \mathbb{C}^4)  $ be a critical point of $ \mathcal{I}^{(m)} $ on  $\mathcal{X}_W $ at a positive level , that is 
$d\mathcal{I}^{(m)} (\psi^{(m)})[h]  -  2 \omega(\psi^{(m)}) \RE \langle \psi^{(m)} | h\rangle_{_{L^{2}} } =0$
	    for any $h \in  W \oplus X_{-}(D)$  and  $\mathcal{I}^{(m)} (\psi^{(m)}) >0$.
	Then  there exists $\delta >0$ such that 
	\begin{equation*}
		d^2 \mathcal{I}^{(m)} (\psi^{(m)})[h ; h]  - 2 \omega(\psi^{(m)}) \|h \|^2_{L^2} \leq - \delta \|h \|^2_{H^{1/2}}, 
		\qquad \forall \, h \in T_{\psi^{(m)}} \mathcal{X}_W.
	\end{equation*}
	Hence in particular $\psi^{(m)} $ is a strict local maximum  of   $ \mathcal{I}^{(m)} $  on $\mathcal{X}_W $.

\end{prop}

\begin{proof}

Since $\mathcal{I}^{(m)} (\psi^{(m)}) >0$  then by  Proposition \ref{prop:PalaisSmale}-(iii) we have that $\omega(\psi^{(m)}) >0$.   
By the $U(1)$- invariance, a critical point $\psi \in  \mathcal{X}_W $ (up to a phase factor) has the following form  $ \psi= a  w + \eta$   with  $ a = a( \eta)= \sqrt{1 -   \|\eta\|^2_{L^2} }$ and $\eta = \psi_{-}$. 
    Now,  any     $ h \in  T_{  \psi}    \mathcal{X}_W$    takes the following form:   $ h = da(\eta) [\xi]   w +    \xi  $ with  $ \xi \in  X_{-}(D)$, and  $da(\eta) [\xi]  = -  a^{-1} \RE \langle \eta | \xi \rangle_{L^2}$. 
     We have 
 
  \begin{align*}
      d^2 & \mathcal{I}^{(m)}   (\psi ) [ h  ; h ]   
      =  a^{-1} da(\eta) [\xi] d^2 \mathcal{I}^{(m)}  (\psi ) [  \psi  ;  (da(\eta) [\xi] w -  a^{-1} da(\eta) [\xi]   \eta ) ] \\
  & +  2   d^2 \mathcal{I}^{(m)}  (\psi ) [  da(\eta) [\xi]   w ;  \xi   ]   +  a^{-2} |da(\eta) [\xi] |^2d^2 \mathcal{I}^{(m)}  (\psi ) [\eta ;  \eta] 
        +  d^2 \mathcal{I}^{(m)}  (\psi ) [ \xi  ; \xi ] . \\
          \end{align*}  
   
       Since  
        $ d^2 \mathcal{I}^{(m)}  (\psi) : H^{1/2} (\R^3, \mathbb{C}^4) \times H^{1/2} (\R^3, \mathbb{C}^4)  \to \R$ is given by 
\begin{align*}
	d^2 & \mathcal{I}^{(m)}  (\psi) [h; k] = 2 \RE \langle k_{+} | h_{+}
	\rangle_{H^{1/2}} - 2 \RE \langle k_{-} | h_{-} \rangle_{H^{1/2}} \\
		& - 2 m  \cfstr \int_{\R^3 \times \R^3 } \frac{ \rho_{\psi} (x) \RE (h ,k ) (y) - \ J_{\psi}(x) \cdot \RE (h, \al k  ) (y)} { |x-y| } dx dy \\
& - 4 m \cfstr \int_{\R^3 \times \R^3 } \frac{\RE (\psi, h )(x)\RE(\psi, k )(y) -\RE (\psi, \al h )(x) \cdot \RE (\psi, \al k)(y) } { |x-y | } dx dy, 
\end{align*}
we have in particular 
    \begin{align*}
	d^2 \mathcal{I}^{(m)}  &(\psi) [\psi; h] 
=  2 d \mathcal{I}^{(m)} (\psi) [h]   - 2 \RE \langle \psi_{+} | h_{+} \rangle_{H^{1/2}}  +2  \RE \langle \psi_{-} | h_{-} \rangle_{H^{1/2}}  \\
&  - 2 m \cfstr \int_{\R^3 \times \R^3 } \frac{  \rho_{\psi} (x) \RE (\psi, h )(y)- J_{\psi}(x) \cdot  \RE (\psi, \al h )(y) } { |x-y | }  dx dy .
\end{align*}
Then,  since $| da(\eta) [\xi] |^2 =  -   a^{-1} da(\eta)  [\xi] \RE \langle \eta | \xi \rangle_{L^2} $   and $ h = da(\eta) [\xi]   w +    \xi  $ ,     we have   
    \begin{align*}      
    a^{-1 }  da(\eta) [\xi]  & d^2 \mathcal{I}^{(m)}  (\psi ) [  \psi  ;  (da(\eta) [\xi] w - a^{-1} da(\eta) [\xi]  \eta )]  \\
             = & 4 \omega(\psi)     \|da(\eta) [\xi] w \|^2_{L^2} -  4 \omega(\psi)  \|a^{-1} da(\eta) [\xi]  \eta \|^2_{L^2} \\
                      & - 2 \| da(\eta) [\xi] w \|^2_{H^{1/2}}  -   2 \| a^{-1} da(\eta) [\xi]  \eta \|^2_{H^{1/2}}  \\
              &  - 2 m \cfstr   | da(\eta) [\xi] |^2  \int_{\R^3 \times \R^3 } \frac{  \rho_{\psi} (x) \rho_{ w}(y)- J_{\psi}(x) \cdot  J_{ w} (y) } { |x-y | } dx dy \\
              &  +  2 m \cfstr  a^{-2} |da(\eta) [\xi] |^2 \int_{\R^3 \times \R^3 } \frac{  \rho_{\psi} (x) \rho_{\eta}   (y)- J_{\psi}(x) \cdot  J_{ \eta} (y) } { |x-y | } dx dy \\
             \leq &2 \omega(\psi)  \|h \|^2_{L^2}    - 2 \| da(\eta) [\xi] w \|^2_{H^{1/2}}  -   2 \| a^{-1} da(\eta) [\xi]  \eta \|^2_{H^{1/2}}  \\
   &  - 2 m \cfstr   | da(\eta) [\xi] |^2    \int_{\R^3 \times \R^3 } \frac{  \rho_{\psi} (x) \rho_{ w}(y)- J_{\psi}(x) \cdot  J_{w} (y) } { |x-y | } dx dy \\
&  +  2 m \cfstr  a^{-2} |da(\eta) [\xi] |^2 \int_{\R^3 \times \R^3 } \frac{  \rho_{\psi} (x) \rho_{  \eta}   (y)- J_{\psi}(x) \cdot  J_{ \eta} (y) } { |x-y | } dx dy. 
   \end{align*}

   Finally we get 
    \begin{align*}
      d^2 & \mathcal{I}^{(m)}   (\psi ) [ h  ; h ]   -  2 \omega(\psi)  \|h \|^2_{L^2}  
    \leq    - 2 \| da(\eta) [\xi] w \|^2_{H^{1/2}}  -   2 \| a^{-1} da(\eta) [\xi]  \eta \|^2_{H^{1/2}}  \\
   &  - 2 m \cfstr     \int_{\R^3 \times \R^3 } \frac{  \rho_{\psi} (x) \rho_{ h}(y)- J_{\psi}(x) \cdot  J_{h} (y) } { |x-y| } dx dy \\
& + 8   \cfstr  \gamma_{K}  (  \| da(\eta) [\xi] w  \|^2_{H^{1/2}}  +  \| \xi  \|^2_{H^{1/2}} )    -  2 (1- 2 \cfstr \gamma_{K} ) \| \xi \|^2_{H^{1/2} }  \\
  &  -  2 (1 -   2 \cfstr \gamma_{K}) \|a^{-1} da(\eta) [\xi] \eta \|^2_{H^{1/2} }  \\
    \leq  &  - 2 (1 -  4   \cfstr  \gamma_{K} ) \| da(\eta) [\xi] w \|^2_{H^{1/2}}   -  2 (1- 6 \cfstr \gamma_{K} ) \| \xi \|^2_{H^{1/2} }  \\
  &  -  4 (1 -    \cfstr \gamma_{K}) \|a^{-1} da(\eta) [\xi] \eta \|^2_{H^{1/2} }  \leq   -  2 (1- 6 \cfstr \gamma_{K} )  \| h \|^2_{H^{1/2} } . \\
   \end{align*}

     	\end{proof}

  For  any $w  \in  \Sigma_{+}$  and  $m \in (0,1]$  we consider   the following  maximization problem
\begin{equation}
\label{eq:sup} 
	\lambda_{W}(m) = \sup_{ \psi \in \mathcal{X}_{W} } \mathcal{I}^{(m)} (\psi). 
\end{equation}

We have the following estimates on $\lambda_{W}(m)$.
\begin{lem}
	\label{lem:apriori_estimates}
 For  any $w  \in  \Sigma_{+}$  and   $m \in (0,1]$,  we have 
	\begin{equation}
		\label{eq:p1}
	(1-  \frac{\cfstr}{2}  \gamma_K) \leq 	(1-  m  \frac{\cfstr}{2}  \gamma_K)  \|w \|^2_{H^{1/2}}  \leq \lambda_{W} (m)  \leq    \| w \|^2_{H^{1/2}}. 
	\end{equation}
\end{lem}
	
\begin{proof}
		
	Clearly $ \lambda_{W} (m) = \sup_{ \psi \in \mathcal{X}_{W}   }
	\mathcal{I}^{(m)} (\psi) \geq \mathcal{I}^{(m)}  (w) $ and by \eqref{eq:correntePositiva}, \eqref{eq:key_estimate} 
	 we have 
    \begin{align*}
		\mathcal{I}^{(m)}  ( w) & \geq   \|w \|^2_{H^{1/2}} - m   \frac{\cfstr}{2}  \int_{\R^3 \times \R^3 } \frac{\rho_{w} (x) \rho_{w}
		(y)  } { |x-y| } dx dy \\
		& \geq (1-  m \frac{\cfstr}{2}  \gamma_K)  \|w \|^2_{H^{1/2}} 
		  \geq (1- \frac{\cfstr}{2}  \gamma_K) >0.
	\end{align*}
		
	Moreover, by  \eqref{eq:stimaGES}   we have
$ \mathcal{I}^{(m)}  (\psi) \leq  \| \psi_+\|^2_{H^{1/2}}  \leq  \| w \|^2_{H^{1/2}}$ for any $\psi \in  \mathcal{X}_{W}$, that is  $ \lambda_{W} (m)\leq    \| w \|^2_{H^{1/2}}$.

\end{proof}

In view of all the above results we  completely solve the maximization problem \eqref{eq:sup},  more precisely we have the following.  
\begin{prop}
	\label{prop:sup_achieved}
	For any $w \in \Sigma_{+}$ and   $m \in (0,1]$   there exists,  unique (up to a phase factor),   $ \psi^{(m)}(w) \in \mathcal{X}_{W}  $,  the strict global
	maximum of $\mathcal{I}^{(m)} $ on $\mathcal{X}_W$, namely 
	\begin{equation*}
		\mathcal{I}^{(m)}  (\psi^{(m)}(w)) = \sup_{ \psi \in \mathcal{X}_{W} }
		\mathcal{I}^{(m)} (\psi) = \lambda_{W}(m).
	\end{equation*}
	Moreover, we have  
	\begin{equation*}
	d \mathcal{I}^{(m)}  (\psi^{(m)}(w))[h] - 2 \omega(\psi^{(m)}(w)) \RE \langle \psi^{(m)}(w) | h \rangle_{L^2} = 0  \qquad  \forall h \in W \oplus X_{-}(D)
	\end{equation*}
	and 
	\begin{itemize}
	
\item[(i)] $0 < \omega(\psi^{(m)}(w))  \leq \lambda_W(m)$ and $ \| \psi_{+}^{(m)}(w) \|^2_{L^2} >  \| \psi_{-}^{(m)}(w) \|^2_{L^2} $;

\medskip

\item[(ii)]   $ \| \psi^{(m)}_{+}(w) \|^2_{H^{1/2}} - \| \psi^{(m)}_{-}(w) \|^2_{H^{1/2}} \geq 1 $;

\medskip

\item[(iii)]  $\| \psi_{-}^{(m)}(w) \|^2_{H^{1/2}} \leq m \frac{\cfstr}{2}  \gamma_K  \|w \|^2_{H^{1/2}} $; 

\medskip		

	\item[(iv)] the map $v \in X_{+}(D) \setminus \{ 0 \}  \to \psi^{(m)}( P(v) )$, with $P(v)  = \| v\|_{L^2}^{-1} v \in \Sigma_{+}$,   is smooth. 
		
		\end{itemize}
		\end{prop}

\begin{proof}

{\it Existence:} Since, by Lemma
	\ref{lem:apriori_estimates}, $ \lambda_{W}(m) >0 $,  by Ekeland's variational principle,  there exists  a  Palais-Smale, maximizing sequence $\{ \psi_{n}^{(m)} \}$ of  $\mathcal{I}^{(m)} $  on $\mathcal{X}_{W}$,  at  a positive level. Then,    by Proposition \ref{prop:PalaisSmale},  $\psi_{n}^{(m)} \to \psi^{(m)}$  in $H^{1/2}$ (up to subsequence), 
 $\omega(\psi_{n}^{(m)}) \to \omega(\psi^{(m)}) >0$ and  $ \| \psi_{+}^{(m)} \|^2_{L^2} >  \| \psi_{-}^{(m)} \|^2_{L^2} $.  Therefore  we conclude that
 
 \begin{equation*}
		\mathcal{I}^{(m)}  (\psi^{(m)}) = \sup_{ \psi \in \mathcal{X}_{W} }
		\mathcal{I}^{(m)} (\psi) = \lambda_{W}(m)
	\end{equation*}
and
 \begin{equation*}
	d \mathcal{I}^{(m)}  (\psi^{(m)})[h] - 2 \omega(\psi^{(m)}) \RE \langle \psi^{(m)} | h \rangle_{L^2} = 0  \qquad \forall h \in W \oplus X_{-}(D).
	\end{equation*}

	 (i)  Note that 
	\begin{align*}
	0 < \omega(\psi^{(m)}) =  \frac{1}{2} d \mathcal{I}^{(m)}  (\psi^{(m)})[\psi^{(m)} ]  
	 \leq  \mathcal{I}^{(m)}  (\psi^{(m)} = \lambda_W(m).
	\end{align*}

		\medskip
	
	(ii) Since   by the $U(1)$- invariance,  we can assume that
	\begin{equation*}
	\psi^{(m)} =  \psi_{+}^{(m)}+ \psi_{-}^{(m)} =(1 - \|\psi_{-}^{(m)} \|^2_{L^2})^{\frac{1}{2}} w + \psi_{-}^{(m)}
	\end{equation*}
	    (up to a phase factor),  with $w = \frac{\psi_{+}}{\| \psi_{+} \|_{L^2} } \in \Sigma_{+}$ then,    since  we have 
	
	\begin{align*}
	 \mathcal{I}^{(m)} (\psi^{(m)})  \geq  \| w \|^2_{H^{1/2}}  - m \frac{\cfstr}{2} \int_{\R^3 \times \R^3} \frac{ \rho_{w}
	(x) \rho_{w} (y) - J_{w}(x)
	\cdot J_{w} (y)} {|x-y |} dx  dy, 
	\end{align*}
	   by lemma \ref{lem:key2}  we get 
	\begin{align*}
	\|  \psi_{+}^{(m)} \|^2_{H^{1/2}} & - \| \psi_{-}^{(m)} \|^2_{H^{1/2}}  - 1 \geq   \| w \|^2_{H^{1/2}} - \|w \|^2_{L^2}  \\
	&  +  m \frac{\cfstr}{2} \int_{\R^3 \times \R^3} \frac{ \rho_{\psi^{(m)}}
	(x) \rho_{\psi^{(m)}} (y)  - J_{\psi^{(m)}}
	(x) \cdot  J_{\psi^{(m)}}(y)} {|x-y|}  dx dy \\
	&  - m \frac{\cfstr}{2} \int_{\R^3 \times \R^3} \frac{ \rho_{w}
	(x) \rho_{w} (y) -  J_{w}
	(x)  \cdot J_{w} (y) } {|x-y |} dx dy \\
	 \geq  & \| w \|^2_{H^{1/2}} - \|w \|^2_{L^2} 
	 -   4 m  \cfstr \gamma_{K}  ( \|w\|^2_{H^{1/2} }- \|w \|^2_{L^2} )  \\
	& - 5 m \cfstr\gamma_{K} (  \|\psi_{-}^{(m)} \|^2_{L^2} \|w\|^2_{H^{1/2} }  + \|  \psi_{-}^{(m)} \|^2_{H^{1/2}}  ).   \\
\geq &  (1- 9  m \cfstr\gamma_{K}   ) (\| w \|^2_{H^{1/2}} - \|w \|^2_{L^2} ) \\
	 &+ 5 m \cfstr \gamma_{K} (  \|\psi_{+}^{(m)}\|^2_{H^{1/2} }  - \|  \psi_{-}^{(m)} \|^2_{H^{1/2}} - 1) 
	\end{align*}
 since $(1- 5 m  \cfstr \gamma_{K} ) >(1- 9 m  \cfstr \gamma_{K} ) >0 $  we may conclude that 
	\begin{equation*}
  \|\psi_{+}^{(m)}\|^2_{H^{1/2} }  - \|  \psi_{-}^{(m)} \|^2_{H^{1/2}} - 1 \geq   0 .
	\end{equation*}

	(iii) Since
	\begin{align*}
	 (1-  m \frac{\cfstr}{2} \gamma_K)  \|w \|^2_{H^{1/2}}   \leq  & \,  \mathcal{I}^{(m)}  (\psi^{(m)})  
	\leq \| \psi_{+}^{(m)}\|^2_{H^{1/2}} - \| \psi_{-}^{(m)}\|^2_{H^{1/2}} 	\\	
 \leq  &\| w \|^2_{H^{1/2}} - \| \psi_{-}^{(m)} \|^2_{H^{1/2}}
	\end{align*}
	we get also $ \| \psi_{-}^{(m)} \|^2_{H^{1/2}} \leq m \frac{\cfstr}{2}  \gamma_K  \|w \|^2_{H^{1/2}}$ .
	
	\medskip
	
{\it Uniqueness:}	
	Suppose we have two different   maximizer $\psi_1^{(m)}, \psi_2^{(m)} \in \mathcal{X}_W$.  We use the Mountain Pass Theorem to reach a contradiction. Indeed, we consider the set
	\begin{equation*}
		\Gamma^{(m)} = \settc{\gamma \colon [0,1]  \to \mathcal{X}_{W}  }{\gamma(0) = \psi_1^{(m)}, \,\, \gamma(1) =
		\psi_2^{(m)} } 
		 \end{equation*}
	and the min-max level
	\begin{equation*}
		c^{(m)} = \sup_{\gamma \in \Gamma^{(m)}} \min_{t \in [0,1]}
		\mathcal{I}^{(m)} (\gamma(t))
	\end{equation*}
We have  $c^{(m)} >0$, indeed,   by the $U(1)$- invariance,  let   $\psi_1^{(m)} = a(\eta_1) w + \eta_1$ and $\psi_2^{(m)} =  a(\eta_2) w + \eta_2$, with 
$\eta_1, \eta_2 \in X_{-}(D)$ and $a( \eta_i)= (1 -   \|\eta_i \| ^2_{L^2} )^{\frac{1}{2} }$ ($i=1,2$),     define 
	$\eta(t) = t \eta_2 + (1-t)\eta_1 \in X_{-}(D) $,  then $g(t) = a(\eta(t)) w + \eta(t) \in \Gamma^{(m)}$.
	
	  Since  $ a(\eta_i)^2  > \frac{1}{2}$ and  by (iii)   we have $\| \eta_i \|^2_{H^{1/2}} \leq m \frac{ \cfstr }{2}\gamma_K\|w \|^2_{H^{1/2}}$  ($i=1,2$),  then for any $t \in [0,1]$ we have 
	\begin{align*}
	\mathcal{I}^{(m)} &(g(t))  \geq   a(\eta(t))^2 \| w \|^2_{H^{1/2}} -   \| \eta(t) \|^2_{H^{1/2}} -  m \frac{\cfstr}{2} \int_{\R^3 \times \R^3 } \frac{\rho_{g(t)} (x) \rho_{g(t)}(y) } { |x-y | } dx dy \\
		  \geq  &(1- m \frac{\cfstr}{2}  \gamma_K)  a(\eta(t))^2  \| w \|^2_{H^{1/2}} - (1 + m \frac{\cfstr}{2}  \gamma_K)   \| \eta(t) \|^2_{H^{1/2}} \\
	 \geq & \,    (1- m  \frac{\cfstr}{2} \gamma_K)  t a(\eta_2)^2 \| w \|^2_{H^{1/2}} -   (1 + m \frac{\cfstr}{2} \gamma_K)  t \| \eta_2 \|^2_{H^{1/2}} \\
	 & + (1- m \frac{\cfstr}{2}  \gamma_K)  (1-t)  a(\eta_1)^2 \| w \|^2_{H^{1/2}} -  (1 + m \frac{\cfstr}{2}  \gamma_K)  (1-t)  \| \eta_1 \|^2_{H^{1/2}} \\
	   \geq &     \frac{1}{2} (1- m \frac{\cfstr}{2}  \gamma_K) \| w \|^2_{H^{1/2}} -  (1 + m \frac{\cfstr}{2}  \gamma_K) m \frac{\cfstr}{2}  \gamma_K \|w \|^2_{H^{1/2}} \\
	      \geq &     \frac{1}{2} (1- 2 \cfstr \gamma_K) \| w \|^2_{H^{1/2}}  >0, 
		\end{align*}
	hence  in particular we get $c^{(m)} \geq \min_{t \in [0,1]} \mathcal{I}^{(m)} (g(t)) >0 $. 

	\medskip
	
	By  Propositions  \ref{prop:concavity} and  \ref{prop:PalaisSmale}-(iii),  we  may conclude  that $c^{(m)}$ is a Mountain
	pass critical level, and that there is   $\phi^{(m)} \in  \mathcal{X}_{W} $,   a Mountain pass critical
	point  for $\mathcal{I}^{(m)}$ on $ \mathcal{X}_{W} $, with $\mathcal{I}^{(m)} (\phi^{(m)}) = c^{(m)} >0$, namely  a
	contradiction with Proposition \ref{prop:concavity}, since a
	Mountain pass critical point cannot be a strict local maximum.

\medskip	
	
(iv) To prove that the map $v \to \psi^{(m)} (P(v))$ is smooth 
	 We use
	the Implicit Function Theorem.    Fix $w_0 \in \Sigma_{+}$  and  let $\psi^{(m)}(w_0) = a(\eta_{0}) w_{0} + \eta_0$ be the   unique maximizer  (up to a phase factor) of $\mathcal{I}^{(m)}$ on $ \mathcal{X}_{W} $. 
	 let  $V \subset X_{+}(D) \setminus \{ 0 \} $ and $U \subset X_{-}(D) $ be,  respectively, the  small neighborhoods of $w_0$ and $ \eta_{0}$,   such that for any $(v, \eta) \in V \times U$ 
	and,  setting $\psi = a(\eta) P(v) + \eta$,  with $a(\eta) = (1- \| \eta \|^2_{L^2})^{1/2}$,   we have 
	 $\| \eta \|^2_{L^2} < \frac{1}{2}  $,   $\mathcal{I}^{(m)}(\psi) >0$ and $ \| \eta  \|^2_{H^{1/2}} \leq m \cfstr \gamma_K  \|P(v)  \|^2_{H^{1/2}} $.

	 We consider the smooth
	maps $F^{(m)} :  V \times U  \to L( X_{-}(D) ) $ given by
	\begin{equation*}
		F^{(m)}(v, \eta)[\xi ] =  d \mathcal{I}^{(m)} ( a(\eta) P(v) +\eta) [da(\eta)[\xi]  P(v) + \xi ]  
	\end{equation*}
	for any  $\xi \in X_{-}(D)$.  Clearly, we have   $P(w_0) = w_0$ and $F^{(m)}(w_0, \eta_0) =  0$.

	The  operator $d_{\eta} F^{(m)}(w_0, \eta_0) :  X_{-} (D)  \to  L( X_{-}(D) ) $ is  given by
	\begin{align*}
		(d_{\eta} F^{(m)} (w_0, \eta_0)[\xi]) [k] = & d^2 \mathcal{I}^{(m)} ( \psi^{(m)}(w_0)) [da(\eta_0)[\xi] w_0+ \xi;  da(\eta_0)[k] w_0+ k ]   \\
		&+ d \mathcal{I}^{(m)}   ( \psi^{(m)}(w_0)) [ d^2 a ( \eta_0) [\xi; k] w_0]  
			 \qquad \forall \xi, k \in X_{-}(D)
	\end{align*}
	
	To prove that $d_{\eta} F^{(m)}(w_0, \eta_0)$ is invertible we apply the Lax- Milgram theorem to the 
	 quadratic form $Q^{(m)} : X_{-}(D)
	\times X_{-}(D) \to \R$ defined by
	\begin{equation*}
		Q^{(m)}[\xi;k] = - (d_{\eta} F^{(m)} (w_0, \eta_0)[\xi]) [k].
	\end{equation*}
	Note that, since 
	\begin{align*}
	 d \mathcal{I}^{(m)}   ( \psi^{(m)}(w_0)) [ d^2 a ( \eta_0) [\xi; \xi] w_0]  &= 2 \omega (\psi^{(m)}(w_0)) a ( \eta_0) d^2 a ( \eta_0) [\xi; \xi]  \\
	 &= -  2 \omega (\psi^{(m)}(w_0)) ( |da(\eta_0)[\xi]|^2  + \|\xi \|^2_{L^2}),  
	 \end{align*}
	 setting $h=  da(\eta_0)[\xi] w_0+ \xi  \in T_{\psi^{(m)}(w_0)} \mathcal{X}_W$, we have 
	\begin{align*}
		Q^{(m)}[\xi; \xi] &= - (d_{\eta} F^{(m)} (w_0, \eta_0)[\xi]) [\xi] \\
		&=  - (d^2 \mathcal{I}^{(m)} ( \psi^{(m)}(w_0)) [h;  h ] - 2\omega (\psi^{(m)}(w_0)) \|h \|^2_{L^2} ).
	\end{align*}
	In view of Proposition \ref {prop:concavity}, there exists $\delta >0$ such that 
	$ Q^{(m)}[\xi;\xi]  \geq \ \delta \|\xi\|^2_{H^{1/2}} $ for any  $ \xi \in X_{-}(D)$. 
	 Hence, by the  Lax-Milgram theorem  we may conclude that 
	   for any $f \in L( X_{-}(D) ) $ there exists
	unique $k  \in X_{-}(D) $ such that $Q^{(m)}[k;\xi] = f[\xi]$ for any $\xi \in X_{-}(D)$,
	namely  such that $d_{\eta} F^{(m)}(w_0, \eta_0)[k] = - f$.
	Finally,  by  the Implicit Function theorem,  there exist  $V_0 \subseteq V  $ and $U_0 \subseteq U$,  neighborhoods,  respectively,   of $w_0$  and $\eta_0$ and a smooth map $\eta^{(m)} : V_0 \to U_{0}$ such
	that $F^{(m)}(v, \eta^{(m)}(v)) = 0 $ for all $v \in V_0$, that is,  $\psi^{(m)} (P(v))  = a(\eta^{(m)}(v)) P(v) + \eta^{(m)}(v)$ is a critical point of $\mathcal{I}^{(m)}$ on 
	$\mathcal{X}_{W}$, with $W = \spann \{ P(v) \}$,  at a positive level. 
        Then, by Proposition \ref{prop:concavity},  
  $\psi^{(m)} (P(v)) $ is a strict local maximum  of $\mathcal{I}^{(m)}$ on $\mathcal{X}_{W}$. 
  
  Again by a contradiction  argument, applying the Mountain Pass theorem as  above, we may conclude that for any $v \in V_0$, 
         $ \psi^{(m)}(P(v))$  is  the unique maximizer  (up to a phase factor) of   $\mathcal{I}^{(m)}$ on  $\mathcal{X}_{W}$, with $W = \spann \{ P(v) \}$.

Moreover, we have that for $w \in \Sigma_{+}$, $d \psi^{(m)}(w):  X_{+}(D) \to X_{+} (D) $ is
	given by
	\begin{align*}
		d\psi^{(m)}(w) [h]  = & a(\psi_{-}(w)) dP(w)[h]  \\
		&+ da(\psi_{-}(w))[d\psi_{-}(w)[dP(w)[h] ] ] w  +  d\psi_{-}(w)[dP(w)[h] ]. 
			\end{align*}
where  $d\psi_{-}(P(v))[dP(v)[h] ] = 	d\eta^{(m)}(v)[h] $ and 
	\begin{equation*}
d\eta^{(m)}(v)[h]  =  - (d_{\eta}F^{(m)}(v, \eta^{(m)}(v)))^{-1} [  d_v F^{(m)}(v, \eta^{(m)}(v))[h]]  , 
				\qquad \forall h \in X_{+}(D) .
	\end{equation*}
\end{proof}

\section {Proof of Theorem \ref{thm:minmax}}
 
In view of the results of Proposition \ref{prop:sup_achieved}   we consider  the smooth functionals $ \mathcal{E}^{(m)} :
X_{+}(D) \setminus \{ 0 \}   \to \R$, for any  $m \in (0,1] $, given by   
\begin{equation*}
	\mathcal{E}^{(m)} (v) = \mathcal{I}^{(m)}  (\psi^{(m)}(P(v)) ) = \sup_{ \psi \in \mathcal{X}_{W} }, 
		\mathcal{I}^{(m)} (\psi) 
\end{equation*} 
where $W = \spann \{ w \} $, with $w= P(v) \in \Sigma_{+}$. We set $\mathcal{E}= \mathcal{E}^{(1)}$  and  
 $\psi(w) = \psi^{(1)}(w)$.

For any  $w \in \Sigma_{+}$ we have
\begin{equation*}
d \mathcal{E}(w)[h] = d \mathcal{I}(\psi(w))[d\psi(w)[dP(w)[h]]], 
\end{equation*}
  and,    setting $k = dP(w)[h],  $
\begin{align*}
	d \psi(w)[k] =  a(\psi_{-}(w)) k 
	+ da(\psi_{-}(w))[d\psi_{-}(w)[k] ] w  +  d\psi_{-}(w)[k]. 
	\end{align*}
Since    $ da(\psi_{-}(w))[d\psi_{-}(w)[k] ] w +  d\psi_{-}(w)[k] \in T_{\psi(w)} \mathcal{X}_W$, by Proposition  \ref{prop:sup_achieved}  we get
 $  d \mathcal{E}(w)[h] =  d \mathcal{I} (\psi(w))[a(\psi_{-}(w)) dP(w)[h] ]$. 
 
  Therefore,  since $dP(w)[h] = h - w \RE \langle w | h \rangle_{L^2}$, we get 
\begin{align*}
 d \mathcal{E}(w)[h] 
 = & d \mathcal{I} (\psi(w)) [ a(\psi_{-}(w)) h ]  -   d \mathcal{I} (\psi(w)) [ a(\psi_{-}(w)) w ] \RE \langle w  | h \rangle_{L^2} \\
 =  & d \mathcal{I} (\psi(w)) [ a(\psi_{-}(w)) h ] -  2 \omega(\psi(w))a(\psi_{-}(w))^2 \RE \langle w  | h \rangle_{L^2} \\
 & =  a(\psi_{-}(w)) (d \mathcal{I} (\psi(w)) [  h ] -  2 \omega(\psi(w)) \RE \langle \psi_{+}(w) | h \rangle_{L^2} )
 \end{align*}
for all $ h \in X_{+}(D)$. 
Since the tangent
space of $\Sigma_{+}  $ at  $w \in \Sigma_{+}$
is the space
\begin{equation*}
	T_{w} \Sigma_{+} = \settc{ h \in X_{+}(D) } { \RE \langle w | h \rangle_{L^2} = 0 }, 
\end{equation*}
and   $ d \mathcal{E}(w)[w] = 0$,  clearly  $\Sigma_{+}$ is   a natural constraint for  $\mathcal{E}$. Therefore  we may conclude that if $w \in \Sigma_{+}$ is a critical point for $\mathcal{E}$ then 
$\psi(w) = a( \psi_{-}(w) ) w +  \psi_{-}(w) $  (as given in  Proposition  \ref{prop:sup_achieved} )  is a critical point for  $\mathcal{I}$  on $\Sigma$, namely 
\begin{equation*}
d \mathcal{I}(\psi(w))[h]  - 2 \omega(\psi(w)) \RE \langle \psi(w) | h \rangle_{L^2} = 0,  \qquad \forall h \in H^{1/2}(\R^3, \C^4).
\end{equation*}
 
 For any $m \in (0,1]$,  we   define the minimization problem 
\begin{equation}
\label{eq:mimization}
	e (m)  = \inf_{ \substack{ W \subset X_{+} \\ \dim W = 1}} \,
	\sup_{ \psi \in \mathcal{X}_{W}} \mathcal{I}^{(m)}  (\psi) =
	\inf_{ w \in \Sigma_{+}} 
	\mathcal{E}^{(m)} (w), 
 \end{equation}
and  $E(m) = m \,  e(m)$, clearly $E = E(1) = e(1)$.  

We have the following estimates on $e(m)$.  
    \begin{lem}
   \label{lem:key_negativa}
   For any $m \in (0,1]$ we have $0 <  e(m) < 1$.
\end{lem}

\begin{proof}
By Lemma \ref{lem:apriori_estimates} we have   $ \lambda_{W}(m) = \sup_{ \psi \in \mathcal{X}_{W}} \mathcal{I}^{(m)}  (\psi)  \geq (1-  \frac{\cfstr}{2}  \gamma_K)  $, hence in particular we get 
$e(m) \geq (1-  \frac{\cfstr}{2}  \gamma_K)  >0$,  for any $m \in (0,1]$. 

\medskip

Now,  since $\Lambda_{+} (D) =   \frac{1}{2} \fw^{-1}  ( \I_{4} - \beta) \fw $,  we  consider  $w = \fw^{-1} \left(  \begin{matrix} 0 \\ v  \end{matrix} \right)  \in \Sigma_{+}$,
 with $v \in H^{1}(\R^3, \C^2)$  and $\| v \|^2_{L^2} = 1$.  In view of Lemma \ref{lem:key2}, 
 since  $\| w \|^2_{H^{1/2}}  = \| v\|^2_{H^{1/2}} $  and $0 \leq  \| w \|^2_{H^{1/2}} - \|w \|^2_{L^2}  \leq \frac{1}{2} \| \nabla v \|^2_{L^2} $,    for any $\psi \in \mathcal{X}_W$  we have

\begin{align*}
	\mathcal{I}^{(m)} (\psi)   = &  \| \psi_{+} \|^2_{H^{1/2}} - \| \psi_{-}
	\|^2_{H^{1/2}}  
	-  m \frac{\cfstr}{2} \int_{\R^3 \times \R^3} \frac{ \rho_{\psi}
	(x) \rho_{\psi} (y)
 - J_{\psi}(x)
	\cdot J_{\psi} (y)} {|x-y |} dx dy \\
	 \leq& \| w \|^2_{H^{1/2}} - \|w \|^2_{L^2} - \| \psi_{-} \|^2_{L^{2}} \| w \|^2_{H^{1/2}} - \| \psi_{-}\|^2_{H^{1/2}}  + 1 \\
	&-     m \frac{\cfstr}{2} \int_{\R^3 \times \R^3}   \frac{ \rho_{v}(x)  \rho_{v}(y) } {|x - y|} dx  dy  \\
 & + 4 m \cfstr \gamma_{K}   \| \nabla v \|^2_{L^2}  + 5  m \cfstr \gamma_{K} (   \|\psi_{-} \|^2_{L^2} \|w\|^2_{H^{1/2} }  + \|  \psi_{-} \|^2_{H^{1/2}} )  \\
  \leq& 1+ \frac{1}{2} (1 +  8m  \cfstr \gamma_{K} )  \| \nabla v \|^2_{L^2}   -     m \frac{\cfstr}{2} \int_{\R^3 \times \R^3}   \frac{  \rho_{v}(x)   \rho_{v}(y) } {|x - y|} dx  dy  \\
 &- (1 -5  m \cfstr \gamma_{K} )(   \|\psi_{-} \|^2_{L^2} \|v\|^2_{H^{1/2} }  + \|  \psi_{-} \|^2_{H^{1/2}} )  . \\
 \end{align*}

Now,  for  any $\epsilon >0 $ we  consider    $v_{\epsilon} (x)=  \epsilon^{3/2} v (\epsilon |x|) $ and   $w_{\epsilon} = \fw^{-1} \left(  \begin{matrix}  0 \\  v_{\epsilon}  \end{matrix} \right)  \in \Sigma_{+}$, 
setting   $W_{\epsilon} = \spann{ w_{\epsilon}}$,  then  
	\begin{align*}	
e(m) -1 \leq   \sup_{\psi \in \mathcal{X}_{W_{\epsilon}}}  \mathcal{I}^{(m)}(\psi)  -1   \leq &   \,  \epsilon^2  \|\nabla v \|^2_{L^2} 	- \epsilon  \left(m  \frac{\cfstr}{2} \int_{\R^3 \times \R^3} \frac{ \rho_{v}
	(x) \rho_{v}(y) } {|x - y|} dx dy  \right),
		 \end{align*}	
hence , taking  $\epsilon >0 $ sufficiently small,   we may conclude that $e(m) - 1 < 0$.

\end{proof}

In view of the above  Lemma \ref{lem:key_negativa} and thanks to the estimate in (ii)-Proposition \ref{prop:sup_achieved}  we have the following result,   essential to the discussion of the minimization problem \eqref{eq:mimization} when using a concentration-compactness argument.

 \begin{prop}
\label{prop:subadditivity}
  $ E(m)$ satisfies the  strict  subadditivity condition
 \begin{equation*}
E(m) <   E(m_1) + E(m_2),  
\end{equation*}
for any $m \in (0,1]$  and   $m_1, m_2 \in (0,1) $ such that $   m_1 + m_2 = m $.  
\end{prop}

\begin{proof}

  For any $\theta >1$  and $m \in (0,1)$ such that   $ \theta m \in (0, 1]$, by Proposition  \ref{prop:sup_achieved}-(ii),  for any $w \in \Sigma_{+}$ we have 
\begin{align*} 
 \theta    ( \mathcal{I}^{( \theta m)}(\psi^{(\theta m)} (w))  -  1  ) 
  \leq    \theta^2  ( \mathcal{I}^{(m)} (\psi^{(\theta m)} (w))  - 1 ) 
 	\leq     \theta^2 ( \mathcal{I}^{(m)} (\psi^{(m)} (w))  - 1 ) 
	\end{align*} 
		hence we get 
		$ \theta    (e ( \theta m)   -1 )  \leq      \theta^2  (	e ( m)   -1 ) $. 
		
	Since    by  lemma \ref{lem:key_negativa}    we have  $e ( m)   -1 < 0  $,    we   get $ \theta^2  (	e ( m)   -1 ) <  \theta  (	e (  m)   -1 )$, 
	namely $ \theta  e ( \theta m)     <   \theta  	e (  m)   $. 
	Therefore we may conclude that  $E(\theta m)  < \theta E(m)$.

Then,  for any $m_1, m_2 \in (0,1) $  such that $   m_1 + m_2 = m  \in (0,1] $, setting $\theta_i = \frac{m}{m_i} $,  then $ \theta_i >1 $ and $\theta_i m _i = m \in (0,1]$,  since 
$ \frac{1}{ \theta_{i}} E(\theta_{i} m_{i})  < E(m_{i})$,  for  $i=1,2$,   we may conclude that 
\begin{align*} 
 E( m)  = \frac{m_1}{m}  E( m) + \frac{m_2}{m}  E( m)  <  E(m_1) +  E( m_2).   
\end{align*} 
  		    \end{proof}

Now,   by Ekeland's variational
principle, there exists a Palais-Smale, minimizing sequence $\{ w_n \}
\subset \Sigma_+$, namely $\mathcal{E}(w_n) = \mathcal{I}(\psi(w_n))  \to E$ and $\|
d \mathcal{E}(w_n) \| \to 0$, then by Proposition  \ref{prop:sup_achieved}, the sequence 
 $\psi_n = \psi(w_n)$  satisfies 
 \begin{align*}
\sup_{\| h \|_{H^{1/2}} = 1 } | d \mathcal{I} (\psi_n) [h] -  2 \omega(\psi_n)\RE  \langle \psi_n  | h \rangle_{L^2}  | \to 0 .
\end{align*}
 Since  $ (1 - 4\cfstr \gamma_{K})  + o(1) \leq   \omega(\psi_n)  \leq  E + o(1)$    we have that  $\omega(\psi_n)   \to \omega \in (0,1) $ (up to subsequence)   and   since   $\| \psi_{-}(w_n) \|^2_{H^{1/2}} \leq \frac{\cfstr}{2}  \gamma_K   \|w_n \|^2_{H^{1/2}} $ by 
Lemma  \ref{lem:apriori_estimates} we have 
\begin{equation*}
 1= \|\psi_n \|^2_{L^2} \leq \|\psi_n \|^2_{H^{1/2} } \leq ( 1 + \frac{\cfstr}{2}  \gamma_K)   \|w_n \|^2_{H^{1/2}}  \leq \frac{ 1 + \frac{\cfstr}{2}  \gamma_K}{1-   \frac{\cfstr}{2}  \gamma_K} ( E + o(1)) . 
\end{equation*}
Therefore   $\{ \psi_n \} $ is a   Palais Smale sequence for the functional
\begin{equation*}
\mathcal{I}_{\omega} (\psi)=  \mathcal{I} (\psi) - \omega \|\psi \|^2_{L^2},
\end{equation*} 
 satisfying   
 \begin{equation*}
  0  < \inf_{n} \|\psi_n \|^2_{H^{1/2} }\leq   \sup_{n}  \|\psi_n \|^2_{H^{1/2} } < + \infty.
  \end{equation*}
  By  the classical  concentration-compactness principle (see \cite{Lions84}) we have 
  a precise characterization of the lack of compactness of bounded Palais Smale sequences of $\mathcal{I}_{w}$, as given 
 in Proposition 3.6 in \cite{GeorgievEstebanSere96}, namely  
  there exists a finite integer $p  \geq 1$, and 
$\phi_{1}, \dots , \phi_{p}  \in  H^{1/2}(\R^3, \C^4)$ non trivial critical points of $\mathcal{I}_{\omega} $, with $ \| \phi_{i} \|^2_{L^2} = m_i$,  
and $p$-sequences $\{ x_n^{i} \} \subset  \R^3$  ($i=1, \dots, p$)  such that    $| x_n^{i}  -  x_n^{j} | \to + \infty$ (for $i \not= j$), and, up to subsequence, 
\begin{equation}
\label{eq:cnvergenza_somma}
\| \psi_n - \sum_{i=1}^p \phi_{i}( \, \cdot \, - x_n^{i} ) \|_{H^{1/2}} \to 0 \qquad  \text{as} \quad n \to + \infty.
\end{equation}
hence, in particular, $ 1 = \| \psi_n \|^2_{L^2}   =  \sum_{i=1}^p m_i  .$

 Moreover, by  \eqref{eq:cnvergenza_somma}  we have 
\begin{align*}
& \langle \psi_n  - \sum_{i=1}^p \phi_{i}( \, \cdot \, - x_n^{i} ) |  (\psi_n)_{+}  - (\psi_n)_{-} \rangle_{H^{1/2}}   = o(1) \\
 & \int_{\R^3 \times \R^3} \frac{ \rho_{\psi_n}(x) ( \psi_n (y) , \psi_n  (y)-  \sum_{i=1}^p \phi_{i}( y - x_n^{i} ) )} {|x-y |} dx dy  = o(1)\\
  & \int_{\R^3 \times \R^3} \frac{ J_{\psi_n}(x)  \cdot (\al  \psi_n (y) , \psi_n  (y)-  \sum_{i=1}^p \phi_{i}( y - x_n^{i} ) )} {|x-y |} dx dy  = o(1).
\end{align*}
then,   since   $(\psi_n )_{\pm}( \, \cdot \, + x_n^{i} )  \rightharpoonup  ( \phi_{i})_{\pm} $ weakly in $H^{1/2}$, we get 
  \begin{align*}
  \| (\psi_n)_{+}  \|^2_{H^{1/2}}  & - \| (\psi_n)_{-} \|^2_{H^{1/2}} =  \langle \psi_n  - \sum_{i=1}^p \phi_{i}( \, \cdot \, - x_n^{i} ) |  (\psi_n)_{+}  - (\psi_n)_{-} \rangle_{H^{1/2}} \\
  & +   \sum_{i=1}^p  \langle \phi_{i} |  (\psi_n)_{+} ( \, \cdot \, + x_n^{i} ) \rangle_{H^{1/2}}   -   \sum_{i=1}^p  \langle \phi_{i}  |(\psi_n)_{-}  ( \, \cdot \, + x_n^{i} )\rangle_{H^{1/2}} \\
  =& \sum_{i=1}^p  \left( \| ( \phi_{i} )_{+}  \|^2_{H^{1/2}}   - \| ( \phi_{i} )_{-} \|^2_{H^{1/2}}  \right) + o(1), 
   \end{align*}
   and, by \eqref{eq:convergenceto0},
   \begin{align*}
   \int_{\R^3 \times \R^3}  \frac{ \rho_{\psi_n}
	(x) \rho_{\psi_n} (y) } {|x-y|}  dx dy  = & \int_{\R^3 \times \R^3} \frac{ \rho_{\psi_n}(x) ( \psi_n (y) , \psi_n  (y)-  \sum_{i=1}^p \phi_{i}( y - x_n^{i} ) )} {|x-y |} dx  dy \\
	&+  \sum_{i=1}^p   \int_{\R^3 \times \R^3} \frac{ \rho_{\psi_n}(x+ x_n^{i}) ( \psi_n(y + x_n^{i} ), \phi_{i}  (y)) } {|x - y|}  dx dy \\
	 =&   \sum_{i=1}^p   \int_{\R^3 \times \R^3} \frac{ \rho_{\phi_{i}}(x)  \rho_{\phi_{i}} (y) } {|x-y |} dx  dy  + o(1), \\
    \end{align*}
    analogously,  
     \begin{align*}
   \int_{\R^3 \times \R^3} &  \frac{ J_{\psi_n}
	(x)  \cdot J_{\psi_n} (y) } {|x-y|} dx dy \\
	 = & \int_{\R^3 \times \R^3} \frac{ J_{\psi_n}(x)  \cdot (\al  \psi_n (y) ,  \psi_n  (y)-  \sum_{i=1}^p \phi_{i}( y - x_n^{i} ) )} {|x - y|} dx dy \\
	&+  \sum_{i=1}^p   \int_{\R^3 \times \R^3} \frac{ J_{\psi_n}(x+ x_n^{i})  \cdot ( \psi_n(y + x_n^{i} ), \al  \phi_{i}  (y)) } {|x-y|} dx dy \\
	& =   \sum_{i=1}^p   \int_{\R^3 \times \R^3} \frac{ J_{\phi_{i}}(x) \cdot J_{\phi_{i}}  (y) } {|x-y |} dx dy  + o(1). \\
    \end{align*}

  which implies $ \mathcal{I}(\psi_n) =  \sum_{i=1}^p   \mathcal{I}(\phi_{i})  + o(1)$ and hence $E =  \sum_{i=1}^p   \mathcal{I}(\phi_{i}) $.

For any $i=1, \dots, p$, we define  $\psi_{i}  = \frac{\phi_{i}}{\|\phi_{i}\|_{L^2}} = \frac{\phi_{i}}{ \sqrt{m_{i}} }  \in \Sigma$ then we have 
\begin{equation*}
 \mathcal{I}(\phi_{i})  =  \mathcal{I}( \sqrt{m_{i}}  \psi_{i})  = m_{i}  \mathcal{I}^{(m_{i})} (  \psi_{i}) 
\end{equation*}
and
\begin{equation*}
0 = d  \mathcal{I}_{\omega}(\phi_{i}) [h] = \sqrt{ m_{i}} ( d\mathcal{I}^{(m_{i})}  (\psi_{i}) [h] -  2 \omega  \RE \langle \psi_{i} | h \rangle_{L^2} ),  \qquad \forall h \in H^{1/2}(\R^3, \C^4),  
 \end{equation*}
 namely, for any $i=1, \dots, p$, $\psi_{i} \in \Sigma$ is a critical point of $\mathcal{I}^{(m_{i})} $ on  $\Sigma$, with  the Lagrange multiplier $\omega \in (0,1)$.

Now, we have the following result,  interesting in itself. 

\begin{lem}
\label{lem:caratt_punti_critici}
Let $\psi \in \Sigma$ be such that 
\begin{equation*}
d\mathcal{I}^{(m)}  (\psi) [h] -  2 \omega  \RE \langle \psi | h \rangle_{L^2} =0 ,   \qquad \forall h \in H^{1/2}(\R^3, \C^4), 
\end{equation*}
for  $m \in (0,1] $ and $\omega \in (0,1)$. 
 Setting $ w =  \frac{\psi_{+}} {\|\psi_{+} \|_{L^2}}  \in \Sigma_{+}$,  then  $\psi = \psi^{(m)}(w)$  is the unique (up to a phase factor) maximizer of   $ \mathcal{I}^{(m)}$  on  $\mathcal{X}_{W}$, with $W = \spann \{ w \}$, namely 
 \begin{equation*}
\mathcal{I}^{(m)} (\psi)   = \mathcal{I}^{(m)} ( \psi^{(m)}(w))   = \sup_{ \phi \in \mathcal{X}_{W} } \mathcal{I}^{(m)}(\phi)  = \mathcal{E}^{(m)}(w).
 \end{equation*}
\end{lem}

\begin{proof}
Clearly   $\psi \in \mathcal{X}_{W}$ and it  is a critical point 
	for $\mathcal{I}^{(m)}$  on $\mathcal{X}_{W}$, moreover 
	
	\begin{equation*}
	 \mathcal{I}^{(m)}(\psi) \geq \frac{1}{2} d\mathcal{I}^{(m)}(\psi)[\psi] =  \omega   > 0.
	\end{equation*}
Therefore  by Proposition \ref{prop:concavity}  we have that 
  $\psi $ is a strict local maximum  for $\mathcal{I}^{(m)}$  on $\mathcal{X}_{W}$.   Moreover,  by  \eqref{eq:key_estimate}  and \eqref{eq:stimaGES} we have 
  \begin{align*}
  \|\psi_+\|^2_{H^{1/2}}  \geq & \omega (\|\psi_+\|^2_{L^2} -   \|\psi_-\|^2_{L^2} )  = \frac{1}{2} d\mathcal{I}^{(m)}(\psi)[\psi_{+} - \psi_{-}] =   \|\psi_+\|^2_{H^{1/2}} +   \|\psi_{-}\|^2_{H^{1/2}} \\
  &- m  \cfstr \int_{\R^3 \times \R^3} \frac{\rho_{\psi}(x) ( |\psi_{+} |^2 - |\psi_{-} |^2)(y)}{ |x-y|}  dx dy \\
		& + m  \cfstr \int_{\R^3 \times \R^3} \frac{J_{\psi}(x)
		\cdot( \RE (\psi_{+} , \al  \psi_{+} )  - \RE (\psi_{-} , \al  \psi_{-} ) (y)}{ |x-y|} dx dy \\
		\geq & \| \psi_{+} \|^2_{H^{1/2}} + \| \psi_{-} \|^2_{H^{1/2}} 
	 -  2 m  \cfstr \int_{\R^3 \times \R^3}
		\frac{\rho_{\psi}(x) |\psi_{+} |^2 (y)}{ |x-y|} dx dy \\
		\geq & (1 - 2m \cfstr \gamma_{K})  \| \psi_{+} \|^2_{H^{1/2}} +  \|\psi_{-}  \|^2_{H^{1/2}} , 
 \end{align*}
 that is  $ \|\psi_{-}  \|^2_{H^{1/2}} \leq  2m \cfstr \gamma_{K}  \| \psi_{+} \|^2_{H^{1/2}} $. 

  Now,  suppose that $\psi $ is not the (unique up to a phase factor) maximizer   of $\mathcal{I}^{(m)}$ on  $\mathcal{X}_{W}$, 
  then,  again by a contradiction  argument, applying the Mountain Pass theorem, as  in  the proof of  Proposition \ref{prop:sup_achieved},  
 we  find a contradiction, namely $\psi = \psi^{(m)}(w)$ and we may conclude that 
        
         \begin{equation*}
\mathcal{I}^{(m)} (\psi) =  \mathcal{I}( \psi^{(m)}(w)) = \sup_{ \psi \in \mathcal{X}_{W}} \mathcal{I}^{(m)}(\psi) . 
 \end{equation*}
 
  \end{proof}

Now  in view of Lemma \ref{lem:caratt_punti_critici}, setting $w_i = \frac{ (\psi_i)_{+}}{\| ( \psi_i)_{+} \|_{L^2}}  $ and   $W_{i} = \text{span} \, \{ w_{i} \} $, for any $i=1, \dots, p$,
 we have that  $\psi_{i} =  \psi^{(m_i)}(w_i)$  (as  in Proposition \ref{prop:sup_achieved}) and 

 \begin{equation*}
\mathcal{I}^{(m_i)} (\psi_i) =  \sup_{ \psi \in \mathcal{X}_{W_i}} \mathcal{I}^{(m_i)}(\psi)   = \mathcal{E}^{(m)} (w_{i})  \geq \inf_{ w \in \Sigma_{+}} 
	\mathcal{E}^{(m)} (w) =  e(m_i).
 \end{equation*}
 
Therefore we may conclude that 
\begin{equation*}
E =   \sum_{i=1}^p   \mathcal{I}(\phi_i)  =   \sum_{i=1}^p   m_{i}  \mathcal{I}^{(m_{i})} (  \psi_{i})  \geq   \sum_{i=1}^p   m_{i}   e(m_i) =      \sum_{i=1}^p  E(m_i) 
 \end{equation*}
  a contradiction with the strict subadditivity condition in  Proposition \ref{prop:subadditivity},  unless we have $p=1$,  that is 
  $\psi_n \to \psi_1 $ strongly in  $H^{1/2}$, hence   $\|\psi_1 \|^2_{L^2} =  m_1 = 1$ and   
  \begin{equation*}
\mathcal{I}(\psi_1) = E =  \inf_{\substack{ W \subset X_{+}(D) \\ \dim W = 1}} \,
	 \sup_{\substack{ \phi \in W \oplus X_{-}(D)  \\  \|\phi\|_{L^2} = 1}} \mathcal{I}(\phi).
\end{equation*}
Moreover 
  \begin{equation*}
  d\mathcal{I}  (\psi_{1}) [h] -  2 \omega  \RE \langle \psi_{1} | h \rangle_{L^2}   = 0  \qquad \forall h \in H^{1/2}(\R^3; \mathbb{C}^4),    
 \end{equation*}
namely $(\psi_1, \omega) \in H^{1/2}(\R^3,\mathbb{C}^4) \times (0,1) $ is a weak solution of \eqref{eq:eigenvalue_MD}. 

Finally, to prove that  $E$  is  the lowest positive  critical value  of the  functional  $\mathcal{I}$ on $\Sigma$,   suppose by contrary that there exists 
$0< \lambda < E $ 
and $\phi \in \Sigma$  such that $\mathcal{I}(\phi) = \lambda$ and 
\begin{equation*}
d\mathcal{I} (\phi) [h] -  2 \mu  \RE \langle \phi | h \rangle_{L^2} =0 ,   \qquad \forall h \in H^{1/2}(\R^3, \C^4)
\end{equation*}
with  $\mu \in \R$ the  Lagrange multiplier.  Since $\lambda >0$ we have that   $\mu > 0 $ (see (iii)-Proposition \ref{prop:PalaisSmale}) and clearly  $\mu \leq \lambda <1$.  
Then $\mu \in (0,1)$ and, setting $w =  \frac{\phi_{+}} {\|\phi_{+} \|_{L^2}}  \in \Sigma_{+}$, we apply Lemma \ref{lem:caratt_punti_critici} to conclude that 
  $\phi $  is the unique (up to a phase factor) maximizer of   $ \mathcal{I}$  on  $\mathcal{X}_{W}$, with $W = \spann \{ w \}$, that is 
 $\lambda = \mathcal{I} (\phi)   = \sup_{ \psi \in \mathcal{X}_{W} } \mathcal{I}(\psi)  \geq E$, 
a contradiction. 

\hfill $\Box$

\bigskip

As a byproduct of all the previous results,  with some minor changes,  we obtain  Theorem \ref{thm:polaron}. Let us briefly list the differences.

\medskip 
{\it Sketch of the proof of Theorem \ref{thm:polaron} :}
 In  the Coulomb-Dirac model  the  self-interaction is attractive, so that  we  formulate the (nonlinear)  eigenvalue problem using the operator  $H= -D$. Clearly 
  \begin{equation*}
  \Lambda_{\pm}(H) = \Lambda_{\mp}(D)  \qquad  \text{and} \qquad   X_{\pm}(H)  =  X_{\mp}(D), 
    \end{equation*}
  then one follows the proof of Theorem    \ref{thm:minmax}  simply by   exchanging the role of $\psi_{\pm}$ with $\psi_{\mp}$, 
   indeed  all the variational arguments  and all the lemmata proved can be carried out,  with no other changes, to deal with  the  functional $\mathcal{I}_{_{CD}}$.
  Note that  the term 
\begin{equation*}
\int_{\R^3 \times \R^3} \frac{ J_{\psi} (x)  \cdot J_{\psi}(y) } {|x - y|} dx dy
  \end{equation*}
is not  present in the functional $\mathcal{I}_{_{CD}}$, this in particular implies that  some of the estimates provided are in fact simplified.

\hfill $\Box$

\section*{Appendix : proof of  Lemma \ref{lem:key2}  }

 {\bf  Lemma \ref{lem:key2} } 

{\it For   any $\psi   = \psi_{+} + \psi_{-}  \in  \Sigma $,  let define $w = \frac{\psi_{+}}{ \| \psi_{+} \|_{L^2}}$    we have 
\begin{align*}
\int_{\R^3 \times \R^3}  &  \frac{ \rho_{\psi}
	(x) \rho_{\psi} (y)  -  J_{\psi}
	(x) \cdot  J_{\psi}(y)} {|x - y|}   dx dy  \geq 
	  \int_{\R^3 \times \R^3}   \frac{ \rho_{w}
	(x) \rho_{w} (y) -  J_{w} (x)  \cdot J_{w}(y)} {|x - z|} dx dy \\
	& -  8 \gamma_{K}   ( \|w\|^2_{H^{1/2} }- \|w\|^2_{L^2} )  - 10  \gamma_{K} (   \|\psi_{-} \|^2_{L^2} \|w\|^2_{H^{1/2} }  + \|  \psi_{-} \|^2_{H^{1/2}} ) .
\end{align*}
Moreover, if $v \in H^{1}(\R^3, \C^2)$,  with $\| v \|^2_{L^2} = 1$,  and  $\frac{\psi_{+}}{ \| \psi_{+} \|_{L^2}} = \fw^{-1} \left(  \begin{matrix} 0 \\ v  \end{matrix} \right)  $ we have 
\begin{align*}
\int_{\R^3 \times \R^3}  &  \frac{ \rho_{\psi}
	(x) \rho_{\psi} (y)  -  J_{\psi}(x) \cdot  J_{\psi}(y)} {|x - y|}   dx dy 
 \geq    \int_{\R^3 \times \R^3}   \frac{ \rho_{v}(x)  \rho_{v}(y) } {|x - y|} dx  dy  \\
 & -  8 \gamma_{K}   \| \nabla v \|^2_{L^2}  - 10  \gamma_{K} (   \|\psi_{-}\|^2_{L^2} \|v \|^2_{H^{1/2} }  + \|  \psi_{-} \|^2_{H^{1/2}} ) .
 \end{align*}}

\begin{proof}
 For any   $ \psi=   \psi_{+} + \psi_{-} \in \Sigma $ with $\psi_{+} =  (1- \|\psi_{-}\|^2_{L^2})^{1/2} w $,  we have 
 \begin{align*}
  \int_{\R^3 \times \R^3}  & \frac{ \rho_{\psi}
	(x) \rho_{\psi} (y) -    J_{\psi}
	(x) \cdot  J_{\psi}(y) } {|x - y|} dx dy 	\\
	=& (1 -  \|\psi_{-}\|^2_{L^2})^2 \int_{\R^3 \times \R^3} \frac{  \rho_{w}(x) \rho_{w} (y) -  J_{w} (x)  \cdot J_{w}(y) } {|x - y|} dx dy  \\
	& + 4 \int_{\R^3 \times \R^3} \frac{   \rho_{\psi_{+} } (x) \RE(\psi_{+}, \psi_{-}) (y) -  J_{\psi_{+} } (x) \cdot \RE(\psi_{+}, \al \psi_{-})(y) } {|x - y|} dx  dy \\
		&+2 \int_{\R^3 \times \R^3} \frac{  \rho_{\psi} (x) \rho_{\psi_{-}} (y) -    J_{\psi}(x) \cdot  J_{\psi_{-}}(y) } {|x - y|}  dx  dy  \\
	&- \int_{\R^3 \times \R^3} \frac{  \rho_{\psi_{-}} (x) \rho_{\psi_{-}} (y) -    J_{\psi_{-}}(x) \cdot  J_{\psi_{-}}(y ) } {|x- y|}  dx dy \\	
	& +4  \int_{\R^3 \times \R^3} \frac{ \RE(\psi_{+}, \psi_{-})(x)  \RE(\psi_{+}, \psi_{-}) (y)} {|x - y|}  dx dy \\
	& - 4  \int_{\R^3 \times \R^3} \frac{ \ \RE(\psi_{+}, \al \psi_{-})(x)  \RE(\psi_{+}, \al\psi_{-}) (y)} {|x - y|} dx  dy, 
	\end{align*}
	hence by  \eqref{eq:key_estimate}, \eqref{eq:positive_Fourier}, \eqref {eq:correntePositiva} and \eqref{eq:stimaGES}   we get
	\begin{align*}
	& \int_{\R^3 \times \R^3}   \frac{\rho_{\psi}
	(x)\rho_{\psi} (y) - J_{\psi}(x) \cdot  J_{\psi}(y) } {|x - y|} 
	 \geq  \int_{\R^3 \times \R^3} \frac{\rho_{w}(x) \rho_{w} (y) -  J_{w} (x)  \cdot J_{w}(y) } {|x - y|}   \\
	&  - 2   \gamma_{K}    \|\psi_{-}\|^2_{L^2}  \| w \|^2_{H^{1/2}}  -  \gamma_{K}   \|\psi_{-}\|^2_{L^2} \| \psi_{-} \|^2_{H^{1/2}}    - 4 \gamma_K  \| \psi_{-} \|_{H^{1/2} }  \|\psi_{-}\|_{L^2}  \| w \|_{H^{1/2}}  \\
& - 4   \big|  \int_{\R^3 \times \R^3}  \frac{   \rho_{w} (x)   \RE(w, \psi_{-}) (y)  } {|x - y|}   \big|   -  4   \big|  \int_{\R^3 \times \R^3}   \frac{  J_{w}(x)  \cdot  \RE (w , \al \psi_{-}) (y)    }  {|x - y|}    \big| .
	\end{align*}
Since $\Lambda_{\pm } (D) =   \frac{1}{2} \fw^{-1}  ( \I_{4} \mp \beta) \fw $,  we  set  $w = \fw^{-1} \left(  \begin{matrix} 0 \\ v  \end{matrix} \right)  $
and $\psi_{-}  = \fw^{-1}  \left(  \begin{matrix} \eta  \\ 0 \end{matrix} \right) $ with $v, \eta \in H^{1/2}(\R^3, \C^2)$ and $\| v \|^2_{L^2} = 1$, in view of Remark \ref{rem:conti_fourier} 
we have 
\begin{align*}
 \big| \int_{\R^3\times \R^3 }  &  \frac{  \rho_{w} (x)   \RE(w, \psi_{-}) (y)  } {|x - y|} dx dy \big|  =   (2 \pi)^{\frac{3}{2}}  \sqrt{\frac{2}{\pi} } \left| \int_{\R^3 } \frac{ \hat{ \rho}_{w} (p)  \mathcal{F} [  \RE(w, \psi_{-}) ](p) }{|p|^2} \, dp \right| \\
&\leq   \sqrt{\frac{2}{\pi} }  \| \rho_{w} \|_{L^1}   \int_{\R^3 } \frac{  | \mathcal{F} [ \RE(w, \psi_{-})  ] (p) |}{|p|^2} \ dp  \\
&\leq    \sqrt{\frac{2}{\pi} }   \int_{\R^3 } \frac{ 1}{|p|^2}  \left(  \frac{1}{ (2 \pi)^{\frac{3}{2}} }  \int_{\R^3}  |(\hat{w} (p-q) , \hat{\psi}_{-}(q)) | \, dq  \right)  dp . 
\end{align*}
Since 
$U^{-1}(p) = u_{+} (p) \I_{4} - u_{-} (p) \beta
    \frac{\al \cdot p}{|p|} $ with $u_{\pm}(p) = \sqrt{\frac{1}{2}(1 \pm \frac{1}{\lambda(p)})}$ we have
 \begin{align*}
  ( \hat{w}(p-q), & \hat{\psi}_{-} (q)) = \left(U^{-1}(p-q)  \left(  \begin{matrix} 0 \\ \hat{v}(p-q)  \end{matrix} \right) , U^{-1}(q)  \left(  \begin{matrix}  \hat{\eta}(q) \\ 0   \end{matrix} \right)  \right)  \\
   =&   - u_{+}(q)   u_{-}(p-q) \frac{p-q}{|p-q|}    \cdot  ( \boldsymbol{ \sigma} \hat{v}(p-q)  , \hat{\eta}(q) )        \\
   &+    u_{+}(p-q) u_{-}(q) \frac{q}{|q|}  \cdot  (\hat{v}(p-q)  , \boldsymbol \sigma  \hat{\eta}(q) )  \\
     =&    u_{-}(p-q) (  u_{+}(p-q)  \frac{q}{|q|} - u_{+}(q)     \frac{p-q}{|p-q|} ) \cdot  ( \boldsymbol{ \sigma} \hat{v}(p-q)  , \hat{\eta}(q) )        \\
  &+    u_{+}(p-q)  ( u_{-}(q)  -   u_{-}(p-q) )  \frac{q}{|q|}  \cdot  (\hat{v}(p-q)  , \boldsymbol \sigma  \hat{\eta}(q) ),  
  \end{align*}
and, since $ u_{-}(q) \geq  \frac{|q|}{2 \lambda(q) } $, we have 
	     \begin{align*}
 |u_{-} (q)-  u_{-}(p-q) |  
 &= \frac{ \frac{1}{2}  |\frac{1}{\lambda(q) } -\frac{1}{\lambda(p-q) }  |}{ u_{-}(q)+ u_{-}(p-q)}  
  \leq \frac{   |\lambda(p-q)  -\lambda(q)  |}{ |q| \lambda(p-q) +   |p-q| \lambda(q)  } \\
& =   \frac{| |q|^2  - |p-q|^2 |}{(|q| \lambda(p) +   |p-q| \lambda(q) ) (\lambda(q) + \lambda(p-q) )}\\
&  \leq    \frac{|p| }{  (\lambda(p-q) + 1 )^{1/2} (\lambda(q) + 1)^{1/2 } }. 
  \end{align*}
Then we get 
 \begin{align*}
 | ( \hat{w}(p-q),   \hat{\psi}_{-} (q))|  
  \leq & 2  u_{-}(p-q) |\hat{v}(p-q) ||\hat{\eta}(q) )|  
  +  |p |  \frac{ |\hat{v}(p-q) |}{(\lambda(p-q) + 1)^{\frac{1}{2} } } \frac{|\hat{\eta}(q) )|}{ (\lambda(q) + 1 )^{\frac{1}{2}}},
   \end{align*}
and  we obtain
\begin{align*}
 \big| \int_{\R^3\times \R^3 }  &  \frac{  \rho_{w} (x)   \RE(w, \psi_{-}) (y)  } {|x - y|}   dx dy \big|  
  \leq   \sqrt{\frac{2}{\pi} }       \int_{\R^3 } \frac{ 2}{|p|^2} \mathcal{F}  \,  \left[ \mathcal{F}^{-1} [u_{-}(p) |\hat{v}|]  \mathcal{F}^{-1} [|\hat{\eta}| ] \right] \, dp\\
    &+  \sqrt{\frac{2}{\pi} }   \int_{\R^3 }  \frac{ 1}{|p|}  \mathcal{F}  \left[ \mathcal{F}^{-1}  \left[ \frac{ |\hat{v}|}{ (\lambda(p) + 1)^{\frac{1}{2}  }} \right] \mathcal{F}^{-1} \left[ \frac{ |\hat{\eta}|}{ (\lambda(p) + 1)^{\frac{1}{2}  }} \right]  \right]  \, dp .
 \end{align*}
Now  by Kato's  inequality we have
 \begin{align*} 
   \sqrt{\frac{2}{\pi} }       \int_{\R^3 }  &\frac{ 2}{|p|^2} \mathcal{F}  \,  \left[ \mathcal{F}^{-1} [u_{-}(p) |\hat{v}|]  \mathcal{F}^{-1} [|\hat{\eta}| ] \right] \, dp
     =    \int_{\R^3 }   \frac{ 2}{|x|}  \mathcal{F}^{-1} [u_{-}(p) |\hat{v}|]  \mathcal{F}^{-1} [|\hat{\eta}| ] \, dx\\
       & \leq 2  \left\|  \frac{  \mathcal{F}^{-1} [u_{-}(p) |\hat{v}|] }{|x|^{\frac{1}{2} }}  \right\|_{L^2}  \left\|   \frac{  \mathcal{F}^{-1} [|\hat{\eta}| ] }{|x|^{\frac{1}{2} }}   \right\|_{L^2} 
           \leq \sqrt{2} \gamma_{K}  \|  (\lambda(p) - 1)^{ \frac{1}{2}  }   \hat{ v} \|_{L^2}  \|   \eta \|_{H^{1/2}} \\
            & \leq \frac{1}{2} \gamma_{K} (\|w\|^2_{H^{1/2}} - \|w \|^2_{L^2}) +  \gamma_{K}  \|  \psi_{-} \|^2_{H^{1/2}},  
 \end{align*}
 since $  \|  (\lambda(p) - 1)^{ \frac{1}{2}  }   \hat{v}  \|^2_{L^2}  = \|v\|^2_{H^{1/2}} - \|v \|^2_{L^2}  = \|w\|^2_{H^{1/2}} - \|w \|^2_{L^2} $.

 Moreover,  by Hardy's   inequality we have
\begin{align*} 
 \sqrt{\frac{2}{\pi} }  & \int_{\R^3 }   \frac{ 1}{|p|}  \mathcal{F}  \left[ \mathcal{F}^{-1} \left[ \frac{ |\hat{v}|}{ (\lambda(p) + 1)^{\frac{1}{2}  }} \right] \mathcal{F}^{-1} \left[ \frac{ |\hat{\eta}|}{ (\lambda(p) + 1)^{\frac{1}{2}  }} \right]  \right]  \, dp \\
 = &   \frac{2}{\pi}   \int_{\R^3 }  \frac{ 1}{|x|^2}    \mathcal{F}^{-1} \left[ \frac{ |\hat{v}|}{ (\lambda(p) + 1)^{\frac{1}{2}  }} \right] \mathcal{F}^{-1} \left[ \frac{ |\hat{\eta}|}{ (\lambda(p) + 1)^{\frac{1}{2}  }} \right]  \, dx \\
 \leq &   \frac{2}{\pi}  \left\|  \frac{ 1}{|x|}    \mathcal{F}^{-1} \left[ \frac{ |\hat{v}|}{ (\lambda(p) + 1)^{\frac{1}{2}  }} \right]   \right\|_{L^2}  \left\|   \frac{ 1}{|x|}   \mathcal{F}^{-1} \left[ \frac{ |\hat{\eta}|}{ (\lambda(p) + 1)^{\frac{1}{2}  }} \right] \right\|_{L^2}  \\
    \leq & \frac{8}{ \pi}  \left\| \frac{ |p|  |\hat{v}| }{ (\lambda(p) + 1)^{\frac{1}{2}  } }     \right\|_{L^2}  \left\|   \frac{ |p|  |\hat{\eta}| }{ (\lambda(p) + 1)^{\frac{1}{2}  } }    \right\|_{L^2}  
      \leq  \gamma_{K} (\|w\|^2_{H^{1/2}} - \|w \|^2_{L^2} ) +  \gamma_{K}\|  \psi_{-} \|^2_{H^{1/2}}
\end{align*}
since $ \left\| \frac{ |p|  |\hat{v}| }{ (\lambda(p) + 1)^{\frac{1}{2}  } }     \right\|^2_{L^2}  = \|  (\lambda(p) - 1)^{ \frac{1}{2}  }   \hat{v}  \|^2_{L^2}   $ and  $ \gamma_{K} = \frac{\pi}{2}$.

Analogously we have
  \begin{align*}
   \big|  \int_{\R^3 \times \R^3}  & \frac{   \RE (w , \al \psi_{-}) (x) \cdot   J_{w}(y) }  {|x - y|} dx  dy   \big|  \\
 &=    (2 \pi)^{\frac{3}{2}}  \sqrt{\frac{2}{\pi} } \left| \int_{\R^3 } \frac{  \mathcal{F} [  \RE(w, \al  \psi_{-}) ](p) \cdot  \mathcal{F} [  \RE(w, \al  w) ](p)  }{|p|^2} \, dp \right| \\
& \leq   \sqrt{\frac{2}{\pi} }   \| (w , \al \psi_{-}) \|_{L^1}    \int_{\R^3 } \frac{  | \mathcal{F} [ \RE(w,  \al w)  ] (p) |}{|p|^2} \ dp  \\
& \leq   \sqrt{\frac{2}{\pi} }  \|w \|_{L^2}  \|\psi_{-}\|_{L^2}   \int_{\R^3 } \frac{ 1}{|p|^2} \left(  \frac{1}{ (2 \pi)^{\frac{3}{2}} } \int_{\R^3}  | (\hat{w} (p-q) ,  \al \hat{w}(q)) |  \, dq  \right) \ dp ,  
  \end{align*}
and   since 
   \begin{align*}
 |  ( \hat{w}(p-q),  \al \hat{w}(q))| 
& \leq   (u_{+} (p-q)  u_{-}(q)   +  u_{+} (q)  u_{-}(p-q)  )  |\hat{v}(q)| |\hat{v}(p-q) |  \\
 &  \leq   (u_{-}(q)   +   u_{-}(p-q)  )  |\hat{v}(q)| |\hat{v}(p-q) | , 
   \end{align*}
then by Kato's inequality   we have
\begin{align*}
   \big|  \int_{\R^3 \times \R^3}  & \frac{   \RE (w , \al \psi_{-}) (x) \cdot   J_{w}(y) }  {|x - y|}  dx  dy   \big|  \\
 & \leq  2 \sqrt{\frac{2}{\pi} }     \|\psi_{-}\|_{L^2}  \int_{\R^3 } \frac{ 1}{|p|^2}   \mathcal{F}  \,  \left[ \mathcal{F}^{-1} [u_{-}(p) |\hat{v}|]  \mathcal{F}^{-1} [|\hat{v}| ] \right]\ dp  \\
  & \leq  2 \|\psi_{-}\|_{L^2}   \left	\|  \frac{  \mathcal{F}^{-1} [u_{-}(p) |\hat{v}|]  }{|x|^{\frac{1}{2} }} \right\|_{L^2}  \left\|   \frac{  \mathcal{F}^{-1} [|\hat{v}| ]}{|x|^{\frac{1}{2} }}    \right\|_{L^2} \\
 & \leq   \sqrt{2} \gamma_{K}  \|  (\lambda(p) - 1)^{ \frac{1}{2}  }    \hat{v} \|_{L^2}  \|\psi_{-} \|_{L^2}  \|  v \|_{H^{1/2}}   \\
  & \leq   \frac{1}{2} \gamma_{K}   (\|w\|^2_{H^{1/2}} - \|w \|^2_{L^2} )  + \gamma_{K}  \|\psi_{-} \|^2_{L^2}  \|  w \|^2_{H^{1/2}} .  
 \end{align*}    
   Therefore we may conclude that 
\begin{align*}
	 \int_{\R^3 \times \R^3}   & \frac{ \rho_{\psi}
	(x) \rho_{\psi} (y) -    J_{\psi}
	(x) \cdot  J_{\psi}(y) } {|x-y|} dx dy  -  \int_{\R^3 \times \R^3} \frac{  \rho_{w}(x) \rho_{w} (y) -  J_{w} (x)  \cdot J_{w}(y) } {|x-y|} dx dy    \\
	\geq &-  8\gamma_{K}  (\|w\|^2_{H^{1/2}} - |w |^2_{L^2}) - 10\gamma_{K}  \|  \psi_{-} \|^2_{H^{1/2}}   - 10 \gamma_{K}   |\psi_{-}|^2_{L^2}  \|  w \|^2_{H^{1/2}} .  
	\end{align*}
Moreover,  by  \cite[Lemma 2.1]{GeorgievEstebanSere96}  we have that 
\begin{align*}
 \int_{\R^3 \times \R^3} \frac{  \rho_{w}(x) \rho_{w} (y) -  J_{w} (x)  \cdot J_{w}(y) } {|x-y|} dx dy   \geq  \int_{\R^3 \times \R^3}& \frac{ (w, \beta w)(x)  (w, \beta w)(y) } {|x - y|} dx dy .
\end{align*}
Now  for    $w = \fw^{-1} \left(  \begin{matrix} 0 \\ v  \end{matrix} \right) $  with $v \in H^1(\R^3, \C^2)$ and $\|v \|^2_{L^2} = 1$,  we have 

\begin{align*}
 (w, \beta w)(x) 
   = & \left( \left( \begin{matrix}  \mathcal{F}^{-1} [ u_{-}(p) \frac{p  }{|p|}   \cdot \boldsymbol \sigma \hat{v} ]  \\  \mathcal{F}^{-1}  [u_{+}(p) \hat{v} ]  \end{matrix}  \right),   
   \left( \begin{matrix}   \mathcal{F}^{-1} [ u_{-}(p) \frac{p  }{|p|}   \cdot \boldsymbol \sigma \hat{v} ]  \\ - \mathcal{F}^{-1}  [u_{+}(p) \hat{v} ]  \end{matrix}  \right) \right)(x) \\
   = &  | \mathcal{F}^{-1} [ u_{-}(p) \frac{p  }{|p|}   \cdot \boldsymbol \sigma \hat{v} ]|^2(x) -  |  \mathcal{F}^{-1}  [u_{+}(p) \hat{v} ] |^2(x)  = |\xi|^2(x) - |f|^2(x) 
    \end{align*}
 
where $f =  \mathcal{F}^{-1}  [u_{+}(p) \hat{v} ] $ and $ \xi =    \mathcal{F}^{-1} [ u_{-}(p) \frac{p  }{|p|}   \cdot \boldsymbol \sigma \hat{v} ] $, then we have 

\begin{align*}
 \int_{\R^3 \times \R^3}& \frac{ (w, \beta w)(x)  (w, \beta w)(y) } {|x - y|}  dx dy  = 
  \int_{\R^3 \times \R^3} \frac{ |f|^2(x)  |f|^2(y) } {|x - y|}  dx dy \\
  &+    \int_{\R^3 \times \R^3} \frac{ |\xi|^2(x)  |\xi|^2(y) } {|x - y|} dx  dy   - 2 \int_{\R^3 \times \R^3} \frac{ |f|^2(x)  |\xi|^2(y) } {|x - y|}  dx dy \\
  \geq &   \int_{\R^3 \times \R^3} \frac{ |f|^2(x)  |f|^2(y) } {|x - y|}  dx dy  -  2 \int_{\R^3 \times \R^3} \frac{ |f|^2(x)  |\xi|^2(y) } {|x - y|} dx dy \\
    \geq &   \int_{\R^3 \times \R^3} \frac{ |f|^2(x)  |f|^2(y) } {|x - y|}  dx dy  -  2 \gamma_K  \| f\|^2_{L^2}  \| \xi \|^2_{H^{1/2} } . 
 \end{align*}
Morover,  setting $\chi =\mathcal{F}^{-1}  [(1 - u_{+}(p) )\hat{v} ] $,   since   $f = v  -  \chi $, by \eqref{eq:positive_Fourier} we have 
\begin{align*}
   \int_{\R^3 \times \R^3} &  \frac{ |f|^2(x)  |f|^2(y) } {|x - y|}  dx dy  =    \int_{\R^3 \times \R^3}   \frac{ |v|^2(x)  |v|^2(y) } {|x - y|}  dx dy  \\
   & -  4 \int_{\R^3 \times \R^3} \frac{  |v|^2 (x) \RE(v, \chi) (y)  } {|x- y |}  dx dy 
		+2 \int_{\R^3 \times \R^3} \frac{  |f|^2(x) |\chi|^2 (y)  } {|x - y|}  dx dy  \\
	&- \int_{\R^3 \times \R^3} \frac{  |\chi|^2 (x) |\chi|^2 (y) } {|x- y |}  dx dy  	
	 +4  \int_{\R^3 \times \R^3} \frac{ \RE(v, \chi)(x)  \RE(v, \chi) (y)} {|x - y|} dx  dy \\ \geq &   \int_{\R^3 \times \R^3}   \frac{ |v|^2(x)  |v|^2(y) } {|x - y|}  dx dy   -  \int_{\R^3 \times \R^3} \frac{  |\chi|^2 (x) |\chi|^2 (y) } {|x- y |}  dx  dy  \\
  &  -  4 \left|  \int_{\R^3 \times \R^3} \frac{  |v|^2 (x) \RE(v, \chi) (y)  } {|x- y |}  dx dy  \right|. 
  \end{align*}

	Since  by  Hardy's inequality we have
\begin{align*}		
\big|  \int_{\R^3 \times \R^3}  &\frac{  |v|^2 (x) \RE(v, \chi) (y)  } {|x- y |}  dx  dy  \big|  \leq 2  \|v\|^2_{L^2}  \|  \chi \|_{L^2}  \| \nabla v \|_{L^2} ,  
  \end{align*}

then, 	 since $\gamma_K = \frac{\pi}{2}$,  by \eqref{eq:key_estimate}  we get 

\begin{align*}	
 \int_{\R^3 \times \R^3}   \frac{ |f|^2(x)  |f|^2(y) } {|x - y|}  dx  dy 
	&   \geq   \int_{\R^3 \times \R^3}   \frac{ |v|^2(x)  |v|^2(y) } {|x - y|}  dx dy  \\
	   & - \gamma_{K}  \| \chi \|^2_{L^2} \|\chi \|^2_{H^{1/2}}   -   \frac{16}{\pi} \gamma_K      \|  \chi \|_{L^2}  \| \nabla v \|_{L^2}  .
			 	  \end{align*}
				Moreover  we have 
\begin{align*}
\| \xi \|^2_{H^{1/2} }  =& \| \lambda(p)^{1/2} | u_{-}(p) \frac{p  }{|p|}   \cdot \boldsymbol \sigma \hat{v} \|^{2}_{L^2}  = \| \lambda(p)^{1/2}  u_{-}(p) \hat{v} \|^{2}_{L^2}  	\\	
=& \frac{1}{2}  (\| v \|^2_{H^{1/2} }  -   \|v \|^2_{L^2}  )  \leq  \frac{1}{4}  \|\nabla v \|^2_{L^2}
\end{align*}
and  $ \|f\|^2_{L^2}  \leq  \|v\|^2_{L^2} = 1$, and since
\begin{equation*}
		|u_{+} (p) - 1| = \frac{ 1 - u_{+} (p)^2}{1+ u_{+} (p)} = \frac{  u_{-} (p)^2}{1+ u_{+} (p)}  \leq u_{-} (p)^2 \leq \frac{1}{2} \frac{ |p|^2} {\lambda(p) +1 } \leq \frac{1}{2}  |p|
			\end{equation*}
we have $ \|\chi \|^2_{L^2}  \leq \frac{1}{4}  \|v\|^2_{L^2} = \frac{1}{4} $, but   also 
\begin{equation*}
\| \chi \|^2_{L^2}  = \| (u_{+}(p)  - 1) \hat{v} \|^2_{L^2}   \leq \frac{1}{4}  \|  \nabla v \|^{2}_{L^2}  
\end{equation*}
and 
\begin{align*}
\| \chi \|^2_{H^{1/2} }  = &\| \lambda(p)^{1/2} |(u_{+}(p)  - 1)\hat{v} | \|^{2}_{L^2}  \leq \| \lambda(p)^{1/2}  u_{-}(p) \hat{v} \|^{2}_{L^2}  \\
= & \frac{1}{2}  (\| v \|^2_{H^{1/2} }  -   \|v \|^2_{L^2}  )   \leq \frac{1}{4}  \|  \nabla v \|^{2}_{L^2}  .
\end{align*}
Therefore, we may conclude that 
\begin{align*}
 \int_{\R^3 \times \R^3}   \frac{ (w, \beta w)(x)  (w, \beta w)(y) } {|x - y|}  &dx dy     \geq    \int_{\R^3 \times \R^3} \frac{ |f|^2(x)  |f|^2(y) } {|x - y|} dx dy   -  2 \gamma_K   \| \xi \|^2_{H^{1/2} } \\
    \geq &  \int_{\R^3 \times \R^3}   \frac{ |v|^2(x)  |v|^2(y) } {|x - y|}  dx dy 
     - \gamma_{K}  \| \chi \|^2_{L^2} \|\chi \|^2_{H^{1/2}}  \\
     & -   \frac{16}{\pi} \gamma_K   \|  \chi \|_{L^2}  \| \nabla v \|_{L^2}  -  2 \gamma_{K}  \| \xi \|^2_{H^{1/2} } \\
    \geq &   \int_{\R^3 \times \R^3}   \frac{ |v|^2(x)  |v|^2(y) } {|x - y|}  dx dy   -   4 \gamma_{K}   \|  \nabla v \|^{2}_{L^2}  .  \\
	  \end{align*}

\end{proof}

\section*{Acknowledgement}
The author is  grateful to Vittorio Coti Zelati for useful discussions.

\providecommand{\bysame}{\leavevmode\hbox to3em{\hrulefill}\thinspace}
\providecommand{\MR}{\relax\ifhmode\unskip\space\fi MR }
\providecommand{\MRhref}[2]{%
  \href{http://www.ams.org/mathscinet-getitem?mr=#1}{#2}
}
\providecommand{\href}[2]{#2}

    \end{document}